\documentclass[11 pt]{amsart}

\usepackage[latin1]{inputenc}
\usepackage{amsmath,amsthm,amssymb,amscd,color, xcolor,mathtools,hyperref}
\usepackage{bbm}

\colorlet{darkgreen}{green!50!black}

\newtheorem{thm}{Theorem}[section]
 
 \newtheorem*{cor*}{Corollary}
 \newtheorem{lem}[thm]{Lemma}

 \theoremstyle{definition}
 \newtheorem{df}[thm]{Definition}
 \newtheorem{prob}[thm]{Problem}
 \theoremstyle{remark}
 \newtheorem{rem}[thm]{Remark}
 
 \numberwithin{equation}{section}

\def\be#1 {\begin{equation} \label{#1}}
\def\ee{\end{equation}}

\def\sqw{\hbox{\rlap{\leavevmode\raise.3ex\hbox{$\sqcap$}}$%
\sqcup$}}
\def\findem{\ifmmode\sqw\else{\ifhmode\unskip\fi\nobreak\hfil
\penalty50\hskip1em\null\nobreak\hfil\sqw
\parfillskip=0pt\finalhyphendemerits=0\endgraf}\fi}

\newcommand{\R}{\mathbb R}

\newcommand{\N}{\mathbb N}
\newcommand{\Z}{\mathbb Z}

\newcommand{\I}{\mathcal{I}}
\newcommand{\J}{\mathcal{J}}

\newcommand{\mes}{\operatorname{mes}}

\usepackage[foot]{amsaddr}
\title{Bounds for spectral projectors on generic tori}
\author[P. Germain]{Pierre Germain\(^1\)}
\address{\(^1\){Courant Institute of Mathematical Sciences}, {New York University}, {{251 Mercer	Street}, {New York}, {N.Y. 10012-1185}, {USA}}; \href{https://orcid.org/0000-0003-3148-4127}{ORCiD:0000-0003-3148-4127}, \href{mailto:pgermain@cims.nyu.edu}{pgermain@cims.nyu.edu}}
\author[S. L. Rydin Myerson]{Simon L. Rydin Myerson\(^2\)}
\address{\(^2\){Mathematics Institute}, {University of Warwick}, {{Zeeman Building}, {Coventry}, {CV4 7AL}, {United Kingdom}};
\href{https://orcid.org/0000-0002-1486-6054}{ORCiD:0000-0002-1486-6054},
\href{mailto:Simon.Rydin-Myerson@warwick.ac.uk}{Simon.Rydin-Myerson@warwick.ac.uk}, \url{https://maths.fan}
}
\keywords{Spectral projectors, discrete restriction, Lp resolvent estimates, lattice points in convex bodies, quadratic forms}
\thanks{\textbf{Acknowledgements:} PG was supported by the Simons collaborative grant on weak turbulence. SLRM was supported by a Leverhulme Early Career Fellowship, by the Fields Institute
	for Research in Mathematical Sciences, and by the National Science Foundation grant DMS 1363013. The authors are thankful to Yu Deng for insightful discussions at an early stage of this project.}
\thanks{{MSC2020}: {42A45}, {11D75}, {11L07} (42B15, 11H06, 11P21)}

\begin{document}

\begin{abstract}We investigate norms of spectral projectors on thin spherical shells for the Laplacian on generic tori, including generic rectangular tori. We state a conjecture and partially prove it, improving on previous results concerning arbitrary tori.
\end{abstract}

\maketitle

\tableofcontents

\section{Introduction}

\subsection{Boundedness of spectral projectors on Riemannian manifolds}

Given a Riemannian manifold $M$ with Laplace-Beltrami operator $\Delta$, and given some $\lambda \geq 1$ and $0<\delta <1$, let
\begin{equation}
\label{defP}
P_{\lambda,\delta} = P_{\lambda,\delta}^{\chi} = \chi \left( \frac{\sqrt{-\Delta} - \lambda}{\delta} \right)
\end{equation}
where $\chi$ is a cutoff function {}{}{taking values in \([0,1]\)} supported in $[-1,1]$, {}{}{and} equal to $1$ on $[-\frac{1}{2},\frac{1}{2}]$. {}{}{This definition is understood through the functional calculus for the Laplace-Beltrami operator, which is a self-adjoint operator on (complete) Riemannian manifolds.}

A general question is to estimate
$$
{ \| P_{\lambda,\delta}^\chi \|_{L^2 \to L^p}, \qquad \mbox{where $p \in [2,\infty]$}},
$$
the exact choice of $\chi$ being immaterial\footnote{If $\chi$ and $\psi$ are two cutoff functions, then $\| P_{\lambda,\delta}^\chi \|_{L^2 \to L^p} \lesssim \| P_{\lambda,\delta}^\psi \|_{L^2 \to L^p} + \| P_{\lambda-\delta,\delta}^\psi \|_{L^2 \to L^p} + \| P_{\lambda + \delta,\delta}^\psi \|_{L^2 \to L^p}$.}.

The answer to this question is known in the case of the Euclidean space: define
$$
p_{ST} = \frac{2(d+1)}{d-1}, \qquad
\sigma(p) = d - 1 - \frac{2d}{p}.
$$
Then by Stein-Tomas~\cite{Stein, SteinTomas} we have
\begin{equation}
\label{swallow}
\| P_{\lambda,\delta} \|_{L^2 \to L^p} \sim
\left\{
\begin{array}{ll}
\lambda^{{\sigma(p)}/{2}} \delta^{1/2} & \mbox{if $p \geq p_{ST}$}, \\
\lambda^{\frac{d-1}{2} \left( \frac{1}{2} - \frac{1}{p} \right)} \delta^{\frac{(d+1)}{2}\left( \frac{1}{2} - \frac{1}{p} \right)} & \mbox{if $2 \leq p \leq p_{ST}$},
\end{array}
\right.
\end{equation}
{}{}{where we write $A \sim B$ if the two quantities $A$ and $B$ are such that $\frac{1}{C}A \leq B \leq CA$, for a constant $C$ which depends only on the \(d\)}.
The answer is again known in the case of compact Riemannian manifolds when $\delta = 1$ (Sogge~\cite{Sogge}{}{}{, Theorem 5.1.1}), for which
\begin{equation}
\label{boundSogge}
\| P_{\lambda,1} \|_{L^2 \to L^p} \sim
\left\{
\begin{array}{ll}
\lambda^{{\sigma(p)}/2} & \mbox{if $p \geq p_{ST}$}, \\
\lambda^{\frac{d-1}{2} \left( \frac{1}{2} - \frac{1}{p} \right)} & \mbox{if $2 \leq p \leq p_{ST}$}.
\end{array}
\right.
\end{equation}

\subsection{Spectral projectors on tori}

\subsubsection{Different kinds of tori}

From now on, we focus on the case of tori given by the quotient {}{}{$\mathbb{R}^d / (\mathbb{Z} e_1 + \dots + \mathbb{Z} e_d)$}, where $e_1,\dots,e_d$ is a basis of $\mathbb{R}^d$, with the standard metric. This is equivalent to considering the operators
$$
P_{\lambda,\delta} = \chi \left( \frac{\sqrt{-\mathcal{Q}(\nabla)} - \lambda}{\delta} \right) \qquad \mbox{on} \;\; \mathbb{T}^d = \mathbb{R}^d / \mathbb{Z}^d,
$$
where $\nabla$ is the standard gradient operator, and $\mathcal{Q}$ is a quadratic form on $\mathbb{R}^d$, with coefficients $\beta_{ij}$: 
$$
\mathcal{Q}(x) = \sum_{i=1}^d \beta_{ij} x^i x^j \qquad \implies \qquad {}{}{\mathcal{Q}(\nabla) = \sum_{i,j=1}^d \beta_{ij} \partial_i \partial_j}.
$$
Here \((\beta_{ij})\) is a symmetric positive definite real  matrix.
Dispensing with factors of $2\pi$, which can be absorbed in $\mathcal{Q}$, the associated Fourier multiplier has the symbol
$$
\chi \left( \frac{\sqrt{\mathcal{Q}(k)} - \lambda}{\delta} \right).
$$
Standard and rectangular tori correspond to the following particular cases.

\begin{itemize}
\item The \textit{standard torus} corresponds to $(e_i)$ being orthonormal, or $\beta_{ij} = \delta_{ij}$.
\item A \textit{rectangular torus} corresponds to $(e_i)$ being orthogonal, or equivalently to a diagonal quadratic form $\beta_{ij} = \beta_i \delta_{ij}$.
\end{itemize}

We will be concerned in this article with generic tori, which for our purposes are defined as follows.

\begin{df} \label{defgeneric}~
\begin{itemize}
\item 
Consider the rectangular tori with \(\beta_i\in[1,2]\) for each \(i\); we say  a property is true for \textit{generic rectangular tori} if it is true on a set of $(\beta_i)_{1\leq i \leq d}$ with full Lebesgue measure in \([1,2]^d\).
\item
Consider the tori with  $\beta_{ij} = \delta_{ij} + h_{ij}$ for each \(1\leq i,j \leq d\) and some $h_{ij}=h_{ji} \in [-\frac{1}{10d^2} , \frac{1}{10d^2} ]$; we say  a property is true for \textit{generic  tori} if it is true for a set of $(h_{ij})_{1\leq i\leq j \leq d}$ with full Lebesgue measure in \( [-\frac{1}{10d^2} , \frac{1}{10d^2} ]^{d(d+1)/2}\).
\end{itemize}
\end{df}

\subsubsection{The conjecture} It was conjectured in~\cite{GM} that, for an arbitrary torus,
\begin{equation}
\label{conjecturearbitrary}
\| P_{\lambda,\delta} \|_{L^{2} \to L^p} \lesssim 1 + (\lambda \delta)^{\frac{(d-1)}{2} \left( \frac{1}{2} - \frac{1}{p} \right)} + \lambda^{\frac{d-1}{2} - \frac{d}{p}} \delta^{1/2} \qquad \mbox{if $\delta > \lambda^{-1}$;}
\end{equation}
{}{}{here and below we denote $A \lesssim B$ if the quantities $A$ and $B$ are such that $A \leq CB$ for a constant $C$, where \(C\) may depend on the dimension \(d\)}.
This paper also contains new results towards this conjecture, as well as a survey of known results

\medskip

In the present paper, we turn our attention towards generic tori, for which the typical spacing between eigenvalues of \(\sqrt{-\Delta}\) is $\lambda^{1-d}$. {{}{}{Indeed, if $k$ ranges in $[-R,R]^d$, then $\sqrt{\mathcal{Q}(k)}$ takes $(2R)^d$ values in $[-CR,CR]$; if the $\beta_{ij}$ are chosen generically we expect these to distribute approximately uniformly.}} This naturally leads to replacing the above conjecture by the following: for generic tori,
\begin{equation}
\label{conjecturegeneric}
\| P_{\lambda,\delta} \|_{L^{2} \to L^p} \lesssim_{\beta,\epsilon} 1 + (\lambda \delta)^{\frac{(d-1)}{2} \left( \frac{1}{2} - \frac{1}{p} \right)} + \lambda^{\frac{d-1}{2} - \frac{d}{p}} \delta^{1/2} \qquad \mbox{if $\delta > \lambda^{1-d+\epsilon}$}
\end{equation}
{}{}{(here the notation $A \lesssim_\alpha B$ means that the constant $C$ in the relation $A \leq C B$ may depend on the parameter $\alpha$)}.

\subsubsection{Known results if $p=\infty$}\label{sec:known-L-infty}
For $p=\infty$, the problem of bounding $\| P_{\lambda,\delta}\|_{L^1 \to L^\infty}$ is closely related to counting lattice points in ellipsoids, a classical question in number theory. Namely, choosing $\chi = \mathbf{1}_{[-1,1]}$,
$$
\| P_{\lambda,\delta} \|_{L^1 \to L^\infty}  = N(\lambda + \delta) - N(\lambda - \delta),
$$
where $N(\lambda)$ is the counting function associated to the quadratic form $\mathcal{Q}$, defined as the number of lattice points $n \in \mathbb{Z}^d$ such that  $\mathcal{Q}(n) < \lambda^2$.

To leading order, $N(\lambda)$ equals $\operatorname{Vol}(E) \lambda^d$, where $\operatorname{Vol}(E)$ is the ellipsoid $\{\mathcal{Q}(x) < 1\}$; the error term is denoted $P(\lambda)$:
$$
N(\lambda) = \operatorname{Vol}(B_1) \lambda^d + P(\lambda).
$$
For the state of the art regarding \(P(\lambda)\) for any fixed \(\mathcal{Q}\) we refer the reader to the comments after~(1.3) in~\cite{GM}, and to the work of Bourgain-Watt~\cite{BW} giving an improved bound for the standard two-dimensional torus. For generic quadratic forms, {}{}{there are a number of additional results}.
\begin{itemize}
\item For generic diagonal forms, Jarn\'ik~\cite{Jarnik28} showed that $P(\lambda) = O (\lambda^{d/2})$ if $d \geq 4$; a weaker, but more general, result is due to Schmidt~\cite{Schmidt}.
\item Landau~\cite{Landau24} showed that $P(\lambda) = \Omega(\lambda^{\frac{d-1}{2}})$ for generic forms.
\item It has been shown that the average size of the error, say $[\mathbb{E} |P(\lambda)|^2]^{1/2}$ is $O(\lambda^{\frac{d-1}{2}})$, for different types of averaging: over translations of the integer lattice~\cite{Kendall}, over shears~\cite{Kelmer}, and over the coefficients $(\beta_i)$ of a diagonal form~\cite{HIW}.
\item When \(d=2\), Trevisan~\cite{Trevisan} has investigated in more detail the distribution of the normalised error~\(P(\lambda)\lambda^{-1/2}\) when \(\mathcal{Q}\) is chosen at random and \(\lambda\) is large.
\item The quantity \(P(\lambda+\delta)-P(\lambda-\delta)\) has also received attention. In particular it has average size \(O(\sqrt{\delta \lambda^{d-1}}) \)  when averaged over translations of the integer lattice~\cite{CGG}, provided that \(\delta \leq \lambda^{-\frac{d-1}{d+1}-\epsilon}\).
\end{itemize}
These results lead to the conjecture that the correct bound for the error term \(P(\lambda)\) for generic \(\mathcal{Q}\), or for generic diagonal \(\mathcal{Q}\), would be $O(\lambda^{\frac{d-1}{2}})$. Meanwhile for  \(P(\lambda+\delta)-P(\lambda-\delta)\) the corresponding conjecture would be \(O(\sqrt{\delta \lambda^{d-1}})\), at least {}{}{for} \(\delta>\lambda^{1-d+\epsilon}\).

\subsubsection{Known results if $p<\infty$}

After the pioneering work of Zygmund~\cite{Zygmund}, Bourgain~\cite{Bourgain} asked for $L^p$ bounds {}{}{for} eigenfunctions of the Laplacian on the (standard) torus. He conjectured that, if $\varphi$ is an eigenfunction of the Laplacian with eigenvalue $\lambda$, then
\begin{equation}
\label{conjecturebourgain}
\| \varphi \|_{L^p} \lesssim \lambda^{\frac{d}{2} - 1 - \frac{d}{p}} \| \varphi \|_{L^2} \qquad \mbox{if}\; \;p \geq \frac{2d}{d-2},
\end{equation}
which is equivalent to the case $\delta = \lambda^{-1}$ of~\eqref{conjecturearbitrary} for the standard torus. Progress on this conjecture appeared in a series of articles~\cite{Bourgain2,BourgainDemeter1,BourgainDemeter2} culminating in the proof of the $\ell^2$-decoupling conjecture by Bourgain and Demeter~\cite{BourgainDemeter3}, which implies~\eqref{conjecturebourgain} if $p \geq \frac{2(d-1)}{d-3}$ and $d \geq 4$.

Bounds for spectral projectors are essentially equivalent to bounds for the resolvent $(-\Delta + z)^{-1}$. This was the point of view adopted in Shen~\cite{Shen}, Bourgain-Shao-Sogge-Yau~\cite{BSSY}, and Hickman~\cite{Hickman}. Here the goal is to prove a sharp bound when \(p^*=\frac{2d}{d-2}\) and \(\delta\) is sufficiently large.

Finally, the authors of the present paper {}{}{were able to prove the conjecture}~\eqref{conjecturearbitrary} when \(\delta\) is sufficently large by combining $\ell^2$-decoupling with a geometry of numbers argument~\cite{GM}.

To the best of our knowledge, all works concerned with $p<\infty$ address either the case of standard tori, or the general case of arbitrary tori; {}{}{the generic case does not seem to have been considered specifically}. This will be {}{}{a} focus of the present paper.

\subsection{A new result through harmonic analysis}

The conjecture~\eqref{conjecturegeneric} was proved in~\cite{GM} for arbitrary tori and  \(\delta\) not too small, and for generic tori we can improve the range for \(\delta\) as follows.

\begin{thm}
\label{thmharmonicanalysis}
For generic rectangular tori and for generic tori (in the sense of Definition~\ref{defgeneric}), the conjecture~\eqref{conjecturegeneric} is verified if $p>p_{ST}$ and, for some $\epsilon >0$,
$$
\mbox{either $\delta > \max \left( \lambda^{-1}, \lambda^{\frac{p}{p-2} \left[ 1 - \frac{d}{2} + \frac{1}{p} \left( \frac{d^2 - d -2}{d-1} \right) \right] + \epsilon} \right)$} \quad 
\mbox{or $ \lambda^{1 - \frac{d}{2} + \frac{1}{p} \frac{d(d-3)}{d-1}+\epsilon} < \delta < \lambda^{-1}$.}
$$
Namely, under these conditions, there holds for almost all choices of $(\beta_i)_{1\leq i \leq d}$ (for generic rectangular tori) or $(\beta_{ij})_{1\leq i,j \leq d}$ (for generic tori)
$$
\| P_{\lambda,\delta} \|_{L^2 \to L^p} \lesssim _{\beta,\epsilon} \lambda^{\frac{d-1}{2} - \frac{d}{p}} \delta^{1/2}.
$$
\end{thm}

 In the particularly well-behaved case when \(p=\infty\) and we consider generic diagonal tori, the theorem matches the classical result of Jarn\'ik~\cite{Jarnik28} mentioned in the first bullet point in Section~\ref{sec:known-L-infty}, which even promotes the upper bound in the theorem to an asymptotic in that case.

The proof of this theorem will be given in Section~\ref{PTHA}. The idea of this proof is to first express the spectral projector through the Schr\"odinger group. {}{}{First note that the operator $\chi\left(\frac{ \sqrt{-\mathcal{Q}(\nabla)} - \lambda}{\delta}\right)$ can also be written $\chi \left(\frac{ -\mathcal{Q}(\nabla) - \lambda^2}{\lambda \delta} \right)$ by adapting the compactly supported function $\chi$. This in turn can be expressed as
$$
P_{\lambda,\delta} =  \lambda \delta \int_{\mathbb{R}} \widehat{\chi}(\lambda \delta t) e^{-2 \pi i\lambda^2 t} e^{-2\pi i t \mathcal{Q}(\nabla)} \,dt
$$}
and then split off into two pieces{}{}{, corresponding to $|t| \lesssim \lambda^{-1}$ and $|t| \gtrsim \lambda^{-1}$ respectively}
\begin{align*}
P_{\lambda,\delta} & = \int \lambda \delta \widehat{\chi}(\lambda t) \widehat{\chi}(\lambda \delta t) e^{-2\pi i\lambda^2 t} e^{-2\pi i t \mathcal{Q}(\nabla)} \,dt +  \int \lambda \delta [1-\widehat{\chi}(\lambda t)] \widehat{\chi}(\lambda \delta t) e^{-2\pi i\lambda^2 t} e^{-2\pi i t \mathcal{Q}(\nabla)} \,dt \\
& = P_{\lambda,\delta}^{\operatorname{small}} + P_{\lambda,\delta}^{\operatorname{large}}.
\end{align*}

It is easy to see that the operator $P_{\lambda,\delta}^{\operatorname{small}}$ can be written {}{}{in} the form $\delta P_{\lambda,1}$ (after adjusting the cutoff function); in other words, this corresponds to the case $\lambda =1$, to which the universal bounds of Sogge apply.

Turning to the term $P_{\lambda,\delta}^{\operatorname{large}}$, its operator norm will be obtained through interpolation between
\begin{itemize}
\item A bound $L^{p_{ST}'} \to L^{p_{ST}}$, Theorem~\ref{thmpST} below. As noted in~\cite{GM}, is a direct consequence of $\ell^2$ decoupling (and valid for any torus).
\item A bound $L^1 \to L^\infty$, for which genericity will be used. Namely, we will prove in Section~\ref{BoWs} that, generically in $(\beta_{ij})$,
\begin{equation}\label{eq:minorarcs}
\int_{1/N}^T \| \chi ( N^{-2} \Delta) e^{it\Delta} \|_{L^1 \to L^\infty} \,dt \lesssim_{\beta,\epsilon} T N^{\frac{d}{2}+\epsilon}.
\end{equation}
\end{itemize}

One could think of \eqref{eq:minorarcs} as square-root cancellation in $L^1 L^\infty$ since $\| \chi ( N^{-2} \Delta) \|_{L^1 \to L^\infty} \sim N^d$. One could also see this as a minor-arc type bound in the spirit of the circle method; indeed in the case \(p=\infty\) the proof in effect reduces to an application of the Davenport-Heilbronn circle method.

\subsection{An elementary approach for \(p=\infty\) and  \(\delta\) very small}

When \(\delta\) is small enough a more elementary counting argument can be used. 



Our main result there is Theorem~\ref{thm:generic-L-infty} below. {}{}{We first state three particularly simple  \(L^1\to L^\infty\) bounds proved at the start of Section~\ref{sec:geomofnumthm} since they are particularly simple; we will then mention consequences for  \(L^1\to L^p\) bounds.}
\begin{thm}
\label{thmgeometryofnumbers}
For generic tori, and also for generic rectangular tori, the following holds.
If \(\delta < \lambda^{1-2d-\epsilon}\), then
\begin{equation}\label{eq:thm_delta_very_small}
	\lVert P_{\lambda,\delta} \rVert_{L^1\to L^\infty} \lesssim_{\beta,\epsilon} 1.
\end{equation}
If instead  \(a\in\Z\) with \( d\leq a\leq 2d\)  and \(\lambda^{-a}\leq \delta\leq \lambda^{1-a}\), then
\begin{equation}\label{eq:thm_blunt_edge}
	\lVert P_{\lambda,\delta} \rVert_{L^1\to L^\infty} \lesssim_{\beta,\epsilon }
	{\delta^{1-\frac {a+1} {d+a+1}} \lambda^{ d -1+\frac{a+1}{d+a+1}+ \epsilon}}.
\end{equation}
Finally  if \(a\in\Z\) with \( \frac{d-1}{2}\leq a<d\) and \(\lambda^{-a}\leq \delta\leq \lambda^{1-a}\), we have
\begin{equation}\label{eq:thm_blunt_main}
	\lVert P_{\lambda,\delta} \rVert_{L^1\to L^\infty} \lesssim_{\beta,\epsilon }
	\delta^{1- \frac 1 {d+1-a}} \lambda^{d - 1 + {\frac{1}{d+1-a}} + \epsilon}.
\end{equation}
\end{thm}
{}{}{
We remark that for \(\delta \leq \lambda^{-1}\), interpolation with the \(L^{p_{ST}'}\to L^{p_{ST}}\) bound from~\cite{GM} gives \[\lVert P_{\lambda,\delta} \rVert_{L^2\to L^p} \lesssim_{\beta,\epsilon}  \lambda^\epsilon( \lVert P_{\lambda,\delta} \rVert_{L^1\to L_\infty})^{\frac{1}{2}-\frac{p_{ST}}{2p}}.\]This would always fall short of the conjecture~\eqref{conjecturegeneric} for \(p<\infty\), even with an optimal \(L^1\to L_\infty\) bound.}
We highlight a few features of these bounds.

\begin{itemize}
	\item In the setting of Theorem~\ref{thmgeometryofnumbers},  conjecture~\eqref{conjecturegeneric} would give \(\lVert P_{\lambda,\delta} \rVert_{L^1\to L^\infty}\lesssim_{\beta,\epsilon} 1+\delta \lambda^{d-1+\epsilon}\).
	\item Although \eqref{eq:thm_blunt_edge} and \eqref{eq:thm_blunt_main} do not recover \eqref{conjecturegeneric}, they do improve on the best known bounds for \(N(\lambda-\delta)-N(\lambda+\delta)\) coming from the results listed in Section~\ref{sec:known-L-infty}.
	\item  Both \eqref{eq:thm_blunt_main} and \eqref{eq:thm_blunt_edge} are special cases of the stronger estimate \eqref{eq:thm-sharp} below, while \eqref{eq:thm_delta_very_small} has a  short self-contained proof.
	\item We restrict to \(a>\frac{d}{2}-1\) in \eqref{eq:thm_blunt_main} and to \(a\leq 2d\) in \eqref{eq:thm_blunt_main} solely becuase the remaining range is already covered by Theorem~\ref{thmharmonicanalysis} or by \eqref{eq:thm_delta_very_small}, see also \eqref{eq:thm-blunt} below.
	\item When \(a=d\) the bound \eqref{eq:thm_blunt_main} would be trivial, and hence for \(a\geq d\) the bound \eqref{eq:thm_blunt_edge} takes over.
\end{itemize}

In the proof of Theorem~\ref{thm:generic-L-infty} we will use the Borel-Cantelli Lemma to reduce to estimates for moments of \(\|P_{\lambda,\delta}\|\), where the moments are taken over \(\lambda\) and \(\beta\). {}{}{A short computation reduces this to the following problem.
\begin{prob}\label{countingmatricesproblem}
Estimate (from above, or asymptotically) the number of matrices of the form
$$
P=
\left(
\begin{matrix}
m^2_{11}&\cdots&m^2_{1b}\\
\vdots&\ddots&\vdots\\
m^2_{d1}&\cdots&m^2_{db}\\
\lambda^2&\cdots&\lambda^2
\end{matrix}
\right),
$$
where the \(m_{ij}\) are integers, all entries in each row lie in a specified dyadic range, and also for each \(k\) the maximal \(k\times k\) subdeterminant of \(P\) lies in a specified dyadic range.
\end{prob}
In section~\ref{cmwps} we give an upper bound in this counting problem using what is in effect linear algebra, relying on the rather technical Lemma~\ref{lem:vol_I} below.} We are then left with a maximum over all possible choices of the various dyadic ranges, and estimating this maximum will be the most challenging part of the proof.

The bound from Lemma~\ref{lem:vol_I} could  be improved. See Remark~\ref{rem:not-optimal} for one path. Another route concerns the case when all \(\beta_{ij}\) are generic, that is the case of generic tori as opposed to generic rectangular tori. Then one could expand \(P\); in place of the squares \(m_{1i}^2,\dotsc,m_{di}^2\) the \(i\)th column would contain all degree 2 monomials in \(m_{1i},\dotsc,m_{di}\). This should allow a smaller bound.

\section{Notations}

\label{sec:notation}

We adopt the following normalizations for the Fourier series on $\mathbb{T}^d$ and Fourier transform on $\mathbb{R}^d$, respectively:
{}{}{
\begin{align*}
& f(x) = \sum_{k \in \mathbb{Z}^d} \widehat{f}_k e^{2\pi i k \cdot x}, \qquad \qquad \widehat{f}_k = \int_{\mathbb{T}^d} f(x) e^{-2\pi i k \cdot x} \,dx, \\
& \widehat{f}(\xi) = \int_{\mathbb{R}^d} f(x) e^{-2\pi ix \cdot \xi} \,dx, \qquad \qquad f(x) = \int_{\mathbb{R}^d} \widehat{f}(\xi) e^{2\pi ix \cdot \xi} \,dx.
\end{align*}
}
With this normalization,
$$
{}{}{\widehat{fg} = \widehat{f} * \widehat{g} \qquad \mbox{and} \qquad \widehat{f * g} = \widehat{f} \widehat{g},}
$$
and {}{}{the} Parseval and Plancherel theorems, respectively, are given by
$$
{}{}{\| f \|_{L^2(\mathbb{T}^d)} = \| \widehat{f} \|_{\ell^2 (\mathbb{Z}^d)} \qquad \mbox{and} \qquad \| f \|_{L^2(\mathbb{R}^d)} = \| \widehat{f} \|_{L^2 (\mathbb{R}^d)}.}
$$

{}{}
{The operator $m(\sqrt{-\mathcal{Q}(\nabla)})$ can be expressed as a Fourier multiplier
$$
m(\sqrt{-\mathcal{Q}(\nabla)}) f = \sum_k m(\sqrt{\mathcal{Q}(k)}) \widehat{f}_k e^{2\pi i k \cdot x},
$$
or through a convolution kernel
$$
m(\sqrt{-\mathcal{Q}(\nabla)}) f (x) = \int K(x-y) f(y) \,dy, \qquad \mbox{with} \qquad K(z) = \sum_k m(\sqrt{\mathcal{Q}(k)}) e^{2\pi i k \cdot z}.
$$
}

In {}{}{Sections}~\ref{sec:linalg} and~\ref{sec:geomofnumthm} we will often join together several matrices \(A_1,\dotsc,A_n\) with the same number of rows to make a larger matrix, for which we use the notation \((A_1\mid\dotsb\mid A_n)\). We view column vectors as matrices with one column, so that \((A\mid \vec{v})\) is the matrix \(A\) with the vector \(\vec{v}\) added as an extra column on the right.

{}{}{Also in Sections~\ref{sec:linalg} and~\ref{sec:geomofnumthm} we use the following notation relating to subdeterminants. If $k \leq \min (p,q)$ and $M$ is a matrix in $\R^{p\times q}$, we will denote $D_k(M)$ the maximum absolute value of a $k\times k$ subdeterminant of $M$:
$$
D_k(M)
= 
\max_{\substack{ \I\subset \{1,\dotsc, p\} \\ \J \subset \{1,\dotsc,q\} \\ \#\I =\#\J=k }}
\big \lvert \det \, (M_{ij})_{i\in\I,j\in \J}\big\rvert.
$$}
{}{}{We further define $D_0(M)=1$ for ease of notation, and we let $D^{(\ell)}_k(M)$ denote the maximal subdeterminant, when the matrix $M$ is restricted to its first $\ell$ columns:}{}{}{
$$
D^{(\ell)}_k(M)
= 
\max_{\substack{ \I\subset \{1,\dotsc, p\} \\ \J \subset \{1,\dotsc,\ell\} \\ \#\I =\#\J=k }}
\big \lvert \det\, (M_{ij})_{i\in\I,j\in \J}\big\rvert
$$}

Given two quantities $A$ and $B$, we denote $A \lesssim B$ if there exists a constant $C$ such that $A \leq CB$, {}{}{and $A \sim B$ if $A \lesssim B$ and $B \lesssim A$}. If the implicit constant $C$ is allowed to depend on $a,b,c$, the notation becomes $A \lesssim_{a,b,c} B$. In the following, it will often be the case that the implicit constant will depend on $\beta$, and on an arbitrarily small power of $\lambda$, for instance $A \lesssim_{\beta,\epsilon} \lambda^\epsilon B$. When this is clear from the context, we simply write $A \lesssim \lambda^\epsilon B$. {}{}{Implicit constants in this notation may always depend on the dimension \(d\) of the torus that is the object of our study.}

{}{}{Finally, the Lebesgue measure of a set $E$ is denoted $\mes E$.}

\section{Bounds on Weyl sums}

\label{BoWs}

Consider the smoothly truncated Weyl sum (or regularized fundamental solution for the anisotropic Schr\"odinger equation)
$$
K_N(t,x) = \sum_{k\in\mathbb{Z}^d}e^{- 2\pi i(x\cdot k+t \mathcal{Q}(k))}\phi\bigg(\frac{k_1}{N}\bigg) \dots \phi\bigg(\frac{k_d}{N}\bigg)
$$
where $\phi$ is a smooth cutoff function supported on $[-1,1]$, equal to $1$ on $[-\frac{1}{2},\frac{1}{2}]$. 

In dimension one, it becomes
$$
K^{(1)}_N(t,y) = \sum_{k\in\mathbb{Z}}e^{2\pi i(yk+k^2t)}\phi\bigg(\frac{k}{N}\bigg).
$$

\subsection{Bound for small time} For small $t$, the following bound holds on any torus.

\begin{lem}\label{dispersive0}
For any {}{}{non-degenerate} quadratic form $\mathcal{Q}$, if $|t| \lesssim \frac{1}{N}$,
$$
|K_N(t,x)| \lesssim \left( \frac{N}{tN + \frac{1}{N}} \right)^{d/2} Z \left( \frac{|x|}{tN + \frac{1}{N}} \right),
$$
{}{}{where $Z$ decays super-polynomially}.
\end{lem}

\begin{proof} This is immediate on applying Poisson summation followed by stationary phase. \end{proof}

\subsection{The one-dimensional case}

As in{}{}{, for example,} equation (8) of Bourgain-Demeter~\cite{BourgainDemeter1}, we have:

\begin{lem}[Weyl bound] \label{weyl0} 
If $a\in\mathbb{Z} \setminus \{0\}$, $q \in \{ 1,\dots, N \}$, $(a,q)=1$ and $\left|t-\frac{a}{q}\right|\leq \frac{1}{q N}$:
\begin{equation}\label{weyl} \forall y \in \mathbb{R}, \qquad \left| K^{(1)}_N(t,y) \right| \lesssim_\epsilon \frac{N^{1+\epsilon}}{\sqrt{q}( 1 + N | t - \frac{a}{q} |^{1/2})} \lesssim \frac{N^{1+\epsilon}}{\sqrt{q}}.
\end{equation} 
\end{lem}

We now define a decomposition into major and minor arcs: for $Q$ a power of 2, $c_0$ a constant which will be chosen small enough, and $Q \leq c_0 N$,
\begin{align*}
& \Lambda_Q(t) = \sum_{\substack{q \in \mathbb{N} \\ \frac{1}{2} Q \leq q < Q}} \sum_{\substack{a \in \mathbb{Z}^* \\ (a,q) = 1}} \mathbf{1}_{[-1,1]} \left( NQ\left( t-\frac{a}{q} \right)\right) \\
\end{align*}
{}{}{In the definition of $\Lambda_Q$, the integer $a$ is not allowed to be zero; it turns out to be convenient to single out the case $a=0$, by letting
$$
\Lambda_0(t) = \mathbf{1}_{[-\frac{1}{N},\frac{1}{N}]}.
$$}

{}{}{Observe that functions of the type $ \mathbf{1}_{[-1,1]} \left( NQ\left( t-\frac{a}{q} \right)\right)$, with $\frac{1}{2} Q_1 \leq q_1 \leq Q_1 \leq c_0 N$ and $\frac{1}{2} Q_2 \leq q_2 \leq Q_2 \leq c_0 N$ are disjoint if $(a_1,q_1) = (a_2,q_2) =1$ and $\frac{a_1}{q_1} \neq \frac{a_2}{q_2}$. Indeed, if $Q_1 \geq Q_2$,
$$
\left| \frac{a_1}{q_1} - \frac{a_2}{q_2} \right| \geq \frac{1}{q_1 q_2} \geq \frac{1}{Q_1 Q_2} \geq \frac{1}{c_0 N Q_2} \geq \frac{2}{NQ_2} \geq \frac{1}{NQ_1} + \frac{1}{NQ_2}.
$$
Therefore, $\Lambda_0 + \sum_{\substack{Q \in 2^\mathbb{N} \\ Q < c_0 N}} \Lambda_Q$ is the characteristic function of a set: the major acs.}

{}{}{The minor arcs will be the complement, with characteristic function $\rho$. This gives the decomposition, for any $t \in \mathbb{R}$,
\begin{equation}
\label{decomposition}
1 = {}{}{\Lambda_0(t)+}\sum_{\substack{Q \in 2^\mathbb{N} \\ Q < c_0 N}} \Lambda_Q(t) + \rho(t).
\end{equation}}

On the support of each of the summands above, the following bounds are available
\begin{itemize}
\item On the support of $\Lambda_0$, there holds $|t| \leq \frac{1}{N}$, and we resort to the short time bound.
\item If $t$ belongs to the support of $\Lambda_{Q}$, there exists $q \in [\frac{Q}{2},Q]$ such that $|t-\frac{a}{q}| < \frac{1}{qN}$ and Weyl's bound gives $\displaystyle  |K_N^{(1)}(t,y)| \lesssim_\epsilon \frac{N^{1+\epsilon}}{\sqrt{Q}}$.
\item {}{}{On the support of $\rho$, by Dirichlet's approximation theorem, there exists $a \in \mathbb{Z}$ and $q \in \{1,\dots,N \}$ relatively prime such that $|t-\frac{a}{q}| < \frac{1}{qN}$. If $q \sim N$, Weyl's bound gives  $\displaystyle  |K_N^{(1)}(t,y)| \lesssim_\epsilon N^{1/2+\epsilon}$. If $q \ll N$, then, since $t$ does not belong to the support of $\sum_{Q \leq c_0 N} \Lambda_Q$, there holds $|t-\frac{a}{q}| \sim \frac{1}{qN}$, which implies again, by Weyl's bound, $\displaystyle  |K_N^{(1)}(t,y)| \lesssim_\epsilon N^{1/2+\epsilon}$.}
\end{itemize}

\subsection{The case of generic rectangular tori} {}{}{In this subsection, we assume that the tori are rectangular, or, equivalently, that the quadratic form $\mathcal{Q}$ is diagonal. First of all, we learn from the bounds on $K_N^{(1)}$ that, on the support of $\Lambda_{Q_1}(\beta_1 t) \dots \Lambda_{Q_k}(\beta_k t) \rho(\beta_{k+1} t) \dots \rho(\beta_d t)$, with all $Q_i \in 2^\mathbb{N}$ and $Q_i \leq c_0 N$,
\begin{equation}
\label{bound1*}
|K_N(t,x)| = |K_N^{(1)}(\beta_1t,x) \dots K_N^{(1)}(\beta_dt,x)|  \lesssim_\epsilon \frac{N^{\frac{k+d}{2}+\epsilon}}{\sqrt{Q_1 \dots Q_k}}.
\end{equation}}

\begin{lem}
\label{generic} Let $\kappa, \epsilon>0$; then for generic $\beta_1, \dots, \beta_d \in [1,2]^d$ there holds, for $1 \leq k \leq d$, $0\leq T \leq N^\kappa$, and $\epsilon>0$, and for all $Q_i, N$ equal to powers of $2$, with $Q_i \leq N$,
\begin{align*}
& \int_0^T \Lambda_{Q_1}(\beta_1 t) \Lambda_{Q_2}(\beta_2 t) \dots \Lambda_{Q_k}(\beta_k t)\,dt \lesssim_{\epsilon,\kappa,\beta_1 \dots \beta_d} N^\epsilon \frac{Q_1 \dots Q_k}{N^k} T.\end{align*}
\end{lem}

\begin{proof} Without loss of generality, we can choose $\beta_1 = 1$. {}{}{Indeed, if $\beta_1,\dots,\beta_d,\frac{T}{\gamma}$ is changed to $\gamma \beta_1, \dots, \gamma \beta_d,T$, with $\gamma > 0$, the integral in the statement of the lemma changes by the factor $\gamma$.} {}{}{We claim that it suffices to prove that
$$
\int_{1}^{2} \dots \int_{1}^{2} \int_0^T \Lambda_{Q_1}(t) \Lambda_{Q_2}(\beta_2 t) \dots \Lambda_{Q_k}(\beta_k t)\,dt \,d\beta_2 \dots d\beta_k \lesssim
{}{}{\mathbf{1}_{t\geq \frac{1}{4N}}}
 \frac{Q_1 \dots Q_k}{N^{k}} T
.
$$
Indeed the case \(t<\frac{1}{4N}\) of the lemma is then immediate, and the remaining case \(\frac{1}{4N}\leq t \leq N^\kappa\) follows by the Borel-Cantelli lemma as explained in Appendix~\ref{appendix}.
}
{}{}{By definition of $\Lambda_Q$},
$$
\int_{1}^{2} \Lambda_Q(\beta t) \,d\beta = \frac{1}{t} \int_{t}^{2t} \Lambda_Q(y)\,dy = \frac{1}{t} \sum_{\substack{(a,q) = 1 \\ \frac{1}{2}Q<q<Q}}\int_t^{2t} \mathbf{1}_{[-1,1]}\left(NQ \left(y-\frac{a}{q} \right) \right)\,dy 
$$
{}{}{To estimate this integral, we observe first that, if $[t,2t] \cap [\frac{a}{q} - \frac{1}{NQ}, \frac{a}{q} + \frac{1}{NQ}] \neq \emptyset$, then
$$
2t \geq \frac{a}{q} - \frac{1}{NQ} \geq \frac{1}{Q} \left[ 1 - \frac{1}{N} \right],
$$
so that, in particular $t {}{}{{}\geq \frac{1}{4Q}}$. Similarly, one can show that the number of $a$'s such that $[t,2t] \cap [\frac{a}{q} - \frac{1}{NQ}, \frac{a}{q} + \frac{1}{NQ}] \neq \emptyset$ is $\lesssim Qt$. Since furthermore the number of $q$'s in $[\frac{1}{2}Q,Q]$ is $\leq Q$, and since finally the integral of $\mathbf{1}_{[-1,1]}\left(NQ \left(y-\frac{a}{q} \right) \right)$ is $\lesssim \frac{1}{NQ}$, we obtain the estimate
$$
\int_{1}^{2} \Lambda_Q(\beta t) \,d\beta \lesssim
\frac{1}{t} \cdot Qt \cdot Q \cdot \frac{1}{NQ}{}{}{\mathbf{1}_{t\geq \frac{1}{4Q}}}
\lesssim \frac{Q}{N}{}{}{\mathbf{1}_{t\geq \frac{1}{4Q}}}
.
$$}
Then, by Fubini's theorem
\begin{align*}
& \int_{1}^{2} \dots \int_{1}^{2} \int_0^T \Lambda_{Q_1}(t) \Lambda_{Q_2}(\beta_2 t) \dots \Lambda_{Q_k}(\beta_k t)\,dt \,d\beta_2 \dots d\beta_k \\
& \qquad \qquad = \int_0^T \Lambda_{Q_1}(t) \left[ \int_{1}^{2} \Lambda_{Q_2}(\beta_2 t) \,d\beta_2 \right] \dots \left[ \int_{1}^{2} \Lambda_{Q_k}(\beta_k t) \,d\beta_k \right]\,dt \\
& \qquad \qquad \lesssim {}{}{\int_0^T} \Lambda_{Q_1}(t) \frac{Q_2 \dots Q_k}{N^{k-1}}\,dt \lesssim
{}{}{\mathbf{1}_{t\geq \frac{1}{4Q_1}}} \frac{Q_1 \dots Q_k}{N^k} T
.
\end{align*}
\end{proof}

A consequence of this lemma is an $L^1_t L^\infty_x$ estimate on $K_N$.

\begin{lem}[Square root cancellation in $L^1_t L^\infty_x$]
\label{lemsrc1} 
Let $\kappa,\epsilon>0$; then for generic {}{}{$\beta_1, \dots, \beta_d \in [1,2]^d$} there holds, for $N$ a power of 2 and $\frac{1}{N} < T < N^\kappa$,
$$
\frac{1}{T} \int_{1/N}^{T} \| K_N(t,\cdot) \| _{L^\infty} \,dt \lesssim_{\epsilon,\kappa,\beta_1 \dots \beta_d} N^{\frac{d}{2} + \epsilon}.
$$
\end{lem}

\begin{proof} {}{}{The first step is to use the decomposition~\eqref{decomposition} in each of the variables $\beta_1 t,\dots \beta_d t$. Note that, since $t > \frac{1}{N}$ and $\beta_i \geq 1$, the term $\Lambda_0(\beta_i t)$ is always zero. By the inequality~\eqref{bound1*}, for almost any choice of $\beta_1,\dots,\beta_d$, there holds}
\begin{align*}
& \frac{1}{T} \int_{1/N}^{T} \| K_N(t,\cdot) \| _{L^\infty} \,dt \\
& {}{}{\qquad \lesssim_\epsilon \sum_{k=0}^d \sum_{\substack{Q_1, \dots Q_k \leq c_0 N \\ Q_i \in 2^\mathbb{N}}}  \frac 1 T \int_{1/N}^T \Lambda_{Q_1}(\beta_1 t) \dots \Lambda_{Q_k}(\beta_k t) \rho(\beta_{k+1} t) \dots \rho(\beta_d t) \,dt  \frac{N^{\frac{k+d}{2}+\epsilon}}{\sqrt{Q_1 \dots Q_k}}} \\
&{}{}{\qquad \lesssim_{\epsilon,\kappa,\beta_1 \dots \beta_d}  N^\epsilon
\sum_{k=0}^d \sum_{Q_1, \dots Q_k \leq c_0 N} \frac{Q_1 \dots Q_k}{N^k}  \frac{N^{\frac{k+d}{2}}}{\sqrt{Q_1 \dots Q_k}}} \\
& \qquad \lesssim N^{\frac{d}{2} + \epsilon}.\qedhere
\end{align*}
\end{proof}

\subsection{The case of generic tori} 
{}{}{We start with an averaging lemma.}

{}{}{\begin{lem}
For $A \in \mathbb{R}$, $\lambda>0$, and $N \geq 10$,
$$
\int_{-1}^1 \min \left( \frac{1}{\| \lambda h + A \|} , N \right) \,dh \lesssim 
\left\{
\begin{array}{ll}
\log N  & \mbox{if $\lambda \geq 1$} \\
\frac{\operatorname{log} N}{\lambda} & \mbox{if $\lambda \leq 1$},
\end{array}
\right.
$$
where, for a real number $x$, the notation $\| x \|$ stands for $\operatorname{dist}(x,\mathbb{Z}){}{}{=\min_{k\in\mathbb{Z}}|k-x|}$.
\end{lem}}

{}{}{\begin{proof} \label{lemmaaverage}
If $\lambda \geq 1$, the left-hand side can be bounded by the average of the function $\min \left( \frac{1}{\|  h \|} , N\right)$, which equals $\log N$. If $\lambda \leq 1$, the left-hand side is bounded by $\int_{-1}^1 \min \left( \frac{1}{\| \lambda h \|} , N \right) \,dh \lesssim \frac{\operatorname{log} N}{\lambda}$.
\end{proof}}

{}{}{Armed with this lemma, we can now prove the desired square root cancellation result - its proof already appeared in~\cite{CG}, but we include an equivalent version here for the reader's convenience. 
Recall that the measure on nonsingular symmetric matrices we consider is given by $B = \operatorname{Id} + h_{ij}$, where $h$ is a symmetric matrix, all of whose coefficients are independent (besides the symmetry assumption) and uniformly distributed in $\left[ \frac{1}{10d^2} , \frac{1}{10d^2} \right]$.}

{}{}{\begin{lem}[Square root cancellation in $L^1_t L^\infty_x$]\label{lemsrc2} Let $\kappa,\epsilon>0$; then for generic $(\beta_{i,j})$, there holds, for $N$ a power of 2 and $\frac{1}{N} < T < N^\kappa$,
$$
\frac{1}{T} \int_{1/N}^{T} \| K_N(t,\cdot) \| _{L^\infty} \,dt \lesssim_{\epsilon,\kappa,\beta_{ij}} N^{\frac{d}{2} + \epsilon}.
$$
\end{lem}}

{}{}{\begin{proof} By the Borel-Cantelli argument in Appendix~\ref{appendix}, the result would follow from the bound
$$
\int \frac{1}{T} \int_{1/N}^{T} \| K_N(t,\cdot) \| _{L^\infty} \,dt \,dB \lesssim_{\epsilon} N^{\frac{d}{2} + \epsilon}, 
$$
which will now be proved. For $x \in \mathbb{T}^d$ and $t \in \mathbb{R}$, applying Weyl {}{}{differencing} gives
$$
|K_N(t,x)|^2 = \sum_{m,n} \phi \left(\frac{m_1 + n_1}{2N} \right) \dots \phi \left( \frac{m_d + n_d}{2N} \right) \left(\frac{m_1 - n_1}{2N} \right) \dots \phi \left( \frac{m_d - n_d}{2N} \right) e^{2itQ(m,n)},
$$
(where the sum is implicitly restricted to $m_i$, $n_i$ having the same parity). By Abel summation, this implies that
$$
\| K_N(t,\cdot) \|_{L^\infty_x} \lesssim \sum_{|n| \lesssim N} \prod_{i=1}^d \min \left(\frac{1}{\| t Q_i(n) \|} , N \right),
$$
where $Q_i(n) = \sum_j \beta_{ij} n_j$. Combining the Cauchy-Schwarz inequality with the above yields
\begin{align*}
 \int_{1/N}^{T} \| K_N(t,\cdot) \| _{L^\infty} \,dt \,dB & \lesssim
\sqrt{T} \left[ \int \int_{1/N}^{T} \| K_N(t,\cdot) \| _{L^\infty_x}^2 \,dt \,dB \right]^{1/2} \\
& \lesssim \sqrt{T} \left[ \int \int_{1/N}^{T} \sum_{|n| \lesssim N} \prod_{i=1}^d \min \left( \frac{1}{\| t Q_i(n) \|} , N \right) \, dt \, dB \right]^{1/2},
\end{align*}
where \(dB = \prod_{i \leq j} d\beta_{i,j}.\)
We now exchange the order of summation and integration, performing first the integration over $B$. Without loss of generality, assume that $|n_1| \sim |n|$. Note that $tQ_{1,1} (n) = t \sum_j \beta_{1,j} n_j$; therefore, by Lemma~\ref{lemmaaverage}, integrating first $\int \min \left( \frac{1}{\| t Q_1(n) \|} , N \right) \,d\beta_{1,1}$ gives $\log N \langle \frac{1}{t|n|} \rangle$. We integrate next $\int \min \left( \frac{1}{\| t Q_2(n) \|} , N \right) \,d\beta_{1,2}$, giving the same result, and similarly $\int \min \left( \frac{1}{\| t Q_k(n) \|} , N \right) \,d\beta_{1,k}$, for $3 \leq k \leq d$. Finally, $\int \prod_{2 \leq i \leq j} d\beta_{i,j}$ gives $O(1)$. Coming back to the sequence of inequalities above,
\begin{align*}
 \int_{1/N}^{T} \| K_N(t,\cdot) \| _{L^\infty} \,dt \,dB & \lesssim \sqrt{T} \left[ \sum_{|n| \lesssim N} ( \log N)^d \int_{1/N}^T  \langle \frac{1}{t|n|} \rangle^d \,dt \right]^{1/2} \\
& \lesssim T N^{d/2} (\log N)^{d/2}.
\end{align*}
\end{proof}}

\section{Proof of Theorem~\ref{thmharmonicanalysis}}

\label{PTHA}

{}{}{An important element of the proof is the optimal $L^{2} \to L^{p_{ST}}$ bound for spectral projectors. As observed in the previous article by the authors~\cite{GM}, it is a consequence of the $\ell^2$ decoupling bound of Bourgain-Demeter. The statement is as follows.}

{}{}{\begin{thm}\label{thmpST}
If $\lambda > 1$ and $\delta <1$, for any positive definite quadratic form $\mathcal{Q}$,
$$
\| P_{\lambda,\delta} \|_{L^2 \to L^{p_{ST}}} \lesssim (1 + \lambda \delta)^{\frac{1}{p_{ST}}}.
$$
\end{thm}}

{}{}{We now turn to the proof of Theorem~\ref{thmharmonicanalysis}.}

\begin{proof} \underline{Step 1: Allowing more general cutoff functions.} {}{}{Define the spectral projector
$$
P'_{\lambda,\delta} = \zeta \left( \frac{\mathcal{Q}(\nabla) + \lambda^2}{\lambda \delta} \right),
$$
where $\zeta$ is a Schwartz function such that $\zeta(0) > 0$ and $\widehat{\zeta}$ is compactly supported.}

{}{}{We claim that it suffices to prove Theorem~\ref{thmharmonicanalysis} for the spectral projector $P'_{\lambda,\delta}$ instead of $P'_{\lambda,\delta}$. Indeed, assume that $P'_{\lambda,\delta}$ enjoys the bound in this theorem. Since there exists $c>0$ such that $\zeta(x) \geq c \mathbf{1}_{[-c,c](x)}$, the desired bound follows\footnote{Here, we are using that, if $m_1(\Delta)$ and $m_2(\Delta)$ are two Fourier multipliers such that $|m_1| \geq |m_2|$, then $\| m_1(\Delta) \|_{L^2 \to L^p} \geq \| m_2(\Delta) \|_{L^2 \to L^p}$. This follows from Parseval's theorem.} for the operator $\mathbf{1}_{[-c,c]} \left( \frac{\mathcal{Q}(\nabla) + \lambda^2}{\lambda \delta} \right)$. This implies in turn this bound for the operator $\mathbf{1}_{[-a,a]} \left( \frac{\sqrt{-\mathcal{Q}(\nabla)} - \lambda}{\delta} \right)$, for a constant $a>0$. Finally, this implies the desired bound for $P_{\lambda,\delta}$ since $|\chi(x)|$ can be bounded by a finite sum of translates of $\mathbf{1}_{[-a,a]}$.}

{}{}{We now claim that $P'_{\lambda,1}$ enjoys the Sogge bounds~\eqref{boundSogge}, just like $P_{\lambda,1}$. This follows from writing
$$
P_{\lambda,1}' = \sum_{n > -\frac{\delta}{\lambda}} P'_{\lambda,1} \mathbf{1}_{[\lambda+(n-1)\delta, \lambda + n\delta]} \left(\sqrt{-\mathcal{Q}(\nabla)} \right)
$$
and bounding
$$
\| P_{\lambda,1}' \|_{L^2 \to L^p} \lesssim c_n \sum_{n > -\frac{\delta}{\lambda}} \left\|  \mathbf{1}_{[\lambda+(n-1)\delta, \lambda + n\delta]} \left(\sqrt{-\mathcal{Q}(\nabla)} \right) \right\|_{L^2 \to L^p} \;\; \mbox{with} \;\; c_n = \sup_{(n-1)\delta < x < n\delta} \left| \zeta \left( \frac{-x^2 + \lambda^2}{\lambda \delta} \right) \right|.
$$
The rapid decay of $\zeta$ implies that $|c_n| \lesssim n^{-N}$ for any $N$, while $\mathbf{1}_{[\lambda+(n-1)\delta, \lambda + n\delta]}$ enjoys the Sogge bounds. Thus, it is not hard to sum the above series, and deduce that $P_{\lambda,1}'$ also enjoys the Sogge bounds.}

{}{}{By a similar argument, it can be shown that Theorem~\eqref{thmpST} applies to $P'_{\lambda,1}$.}

\medskip

\noindent \underline{Step 2: splitting the spectral projector.}
Writing the function $x \mapsto \zeta \left( \frac{\mathcal{Q}(x) + \lambda^2}{\lambda \delta} \right)$ as a Fourier transform, the operator $P_{\lambda,\delta}'$ becomes
\begin{equation}
\label{integralP'}
P_{\lambda,\delta}' = \int_{\mathbb{R}} \lambda \delta \widehat{\zeta}(\lambda \delta t) e^{-2\pi i \lambda^2 t} e^{-2\pi it\mathcal{Q}(\nabla)} \,dt,
\end{equation}
with the kernel
$$
P_{\lambda,\delta}'(x) = \int_{\mathbb{R}} \lambda \delta \widehat{\zeta}(\lambda \delta t) e^{-2\pi i \lambda^2 t} K_N(t,x) \,dt;
$$
here, we choose $N$ to be a power of $2$ in the range $[2\lambda,4\lambda]$.

The basic idea is to split the integral giving $P_{\lambda,\delta}$ into two pieces, $|t| < \lambda^{-1}$ and $|t| > \lambda^{-1}$. The former corresponds to an operator of the type $\delta P_{\lambda,1}'$, for which bounds are well-known: this corresponds to the classical Sogge theorem. The latter can be thought of as an error term, it will be bounded by interpolation between $p=p_{ST}$ and $p= \infty$, and it is for this term that genericity is used.

Turning to the implementation of this plan, we write
\begin{align*}
P_{\lambda,\delta}' & = \int_{\mathbb{R}} \lambda \delta \widehat{\zeta}(\lambda t) \widehat{\zeta}(\lambda \delta t) e^{-2\pi  i \lambda^2 t} e^{-2\pi it \mathcal{Q}(\nabla)} \,dt + \int_{\mathbb{R}} \lambda \delta [1- \widehat{\zeta}(\lambda t)] \widehat{\zeta}(\lambda \delta t) e^{-2\pi  i \lambda^2 t} e^{-2\pi it\mathcal{Q}(\nabla)} \,dt \\
& = P^{\operatorname{small}}_{\lambda,\delta} +  P^{\operatorname{large}}_{\lambda,\delta} 
\end{align*}

\bigskip

\noindent \underline{Step 3: Bounding the term corresponding to small $t$.} {}{}{Observe that $P_{\lambda,\delta}^{\operatorname{small}}$ can be written $\delta P''_{\lambda, 1}$, where $P''_{\lambda, 1}$ is a variation on $P'_{\lambda,1}$; this can be compared to the definition of $P_{\lambda,\delta}^{\operatorname{small}}$ and~\eqref{integralP'}. We saw in Step 1 that $P'_{\lambda,1}$ enjoys the Sogge bounds, and this remains true for $P''_{\lambda,1}$}.

Furthermore, by a classical $TT^*$ argument, the operator norm of the spectral projector $L^{p'} \to L^p$ is the square of the operator norm of the spectral projector $L^2 \to L^p$ (once again, up to redefining the cutoff function $\chi$).

Therefore, it enjoys the bound
\begin{equation}
\label{Psmall}
{}{}{\| P_{\lambda,\delta}^{\operatorname{small}} \|_{L^{p'} \to L^p} \lesssim \delta \| P''_{\lambda, 1} \|^2_{L^2 \to L^p} \lesssim \lambda^{d-1-\frac{2d}{p}} \delta \qquad \mbox{for $p \geq p_{ST}$}.}
\end{equation}

\bigskip

\noindent \underline{Step 4: Bounding the term corresponding to large $t$.}
In order to bound this term, we will interpolate between

\begin{itemize}

\item {}{}{The case $p=p_{ST}$: in this case, we resort to Theorem~\ref{thmpST}. We saw in Step 1 that it applies to $P'_{\lambda,\delta}$, and, by the same argument, it applies to $P''_{\lambda,\delta}$.} This gives
\begin{align*}
\| P_{\lambda,\delta}^{\operatorname{large}} \|_{L^{p_{ST}'} \to L^p_{ST}} & \lesssim \| P'_{\lambda,\delta} \|_{L^{p_{ST}'} \to L^p_{ST}} + \| P^{\operatorname{small}}_{\lambda,\delta} \|_{L^{p_{ST}'} \to L^p_{ST}} \\
& \lesssim \| P'_{\lambda,\delta} \|_{L^{p'_{ST}} \to L^{p_{ST}}} + \delta \| P''_{\lambda,\delta} \|_{L^{p'_{ST}} \to L^{p_{ST}}} \\
& \lesssim_\epsilon \lambda^\epsilon [1 + (\lambda \delta)^{2/p_{ST}} + \delta \lambda^{2/p_{ST}}] \lesssim \lambda^\epsilon [ 1 + (\lambda \delta)^{2/p_{ST}}].
\end{align*}
\item {}{}{The case $p = \infty$: in this case, we resort to Lemma~\ref{lemsrc1} (generic rectangular tori) and Lemma~\ref{lemsrc2} (generic tori). In order for these lemmas to apply, we add a further requirement on $\zeta$, namely that its Fourier transform be $1$ in a neighbourhood of zero.} Then, for almost any choice of $(\beta_{ij})$,
$$
\| P_{\lambda,\delta}^{\operatorname{large}} \|_{L^1 \to L^\infty} \lesssim \int_{\mathbb{R}} \lambda \delta |1 - \widehat{\zeta}(\lambda t) | \left|\widehat{\zeta}(\lambda \delta t)\right| \| K_{N} (t,\cdot) \|_{L^\infty} \,dt \lesssim_{\beta,\epsilon} \lambda^{\frac{d}{2} + \epsilon}.
$$
\end{itemize}

Interpolating between these two bounds gives for almost any choice of $(\beta_{ij})$
\begin{equation}
\label{Plarge}
\| P_{\lambda,\delta}^{\operatorname{large}} \|_{L^{p'} \to L^p} \lesssim_{\beta,\epsilon} (1 + \lambda \delta)^{2/p} \lambda^{\frac{d}{2} \left( 1 - \frac{p_{ST}}{p} \right) + \epsilon} \qquad \mbox{for $p \geq p_{ST}$}.
\end{equation}

\medskip

\noindent \underline{Step 5: conclusion.} Finally, combining~\eqref{Psmall} and~\eqref{Plarge}, and using that $\| P_{\lambda,\delta}\|_{L^2 \to L^p}^2 \sim \| P_{\lambda,\delta}^2 \|_{L^{p'} \to L^p}$ (by the classical $TT^*$ argument) gives
$$
\| P_{\lambda,\delta}\|_{L^2 \to L^p} \lesssim_{\beta,\epsilon} \lambda^{\frac{d-1}{2} - \frac{d}{p}} \delta^{1/2} + (1 + \lambda \delta)^{\frac{1}{p}} \lambda^{\frac{d}{4} \left( 1 - \frac{p_{ST}}{p} \right) + \epsilon},
$$
from which the desired result follows.
\end{proof}

\section{Some linear algebra}\label{sec:linalg}

{}{}{
In this section we assemble technical tools to attack Problem~\ref{countingmatricesproblem}. Recall that the goal is to count the number of matrices
\[
\left(
\begin{matrix}
	m^2_{11}&\cdots&m^2_{1b}\\
	\vdots&\ddots&\vdots\\
	m^2_{d1}&\cdots&m^2_{db}\\
	\lambda^2&\cdots&\lambda^2
\end{matrix}
\right),
\]
where the \(m_{ij}\) are integers in given dyadic intervals and the maximal subdeterminants of \(P\) also lie in some specified dyadic intervals. The idea is to add the columns one by one, so that we count the number of possible \((m_{11},\dotsc,m_{d1})\), and for each possibility we count the number of \((m_{12},\dotsc,m_{d2})\), and so on. The main goal of this section is Lemma~\ref{lem:vol_I},  which can be understood as an estimate for the \emph{measure} of the \emph{real} vectors \((m_{1k},\dotsc,m_{dk})\) which are within a distance \(O(1)\) of a vector satisfying the required conditions, given the previous columns.
}

{}{}{In this and the next section we will often use the  notation \(D_k(M), D_k^{(\ell)}(M)\) defined in section~\ref{sec:notation}.}

\subsection{Singular values and largest subdeterminants}

{}{}{
We begin with a number of general statements about the size of the subdeterminants of a \(p\times q\) matrix, and their relation to the \emph{singular value decomposition}, a type of canonical form for matrices. Throughout this subsection, implicit constants in \(\lesssim\) and \(\sim\) notation may depend on \(p\) and \(q\).}

\begin{lem}[Singular value decomposition] \label{SVD} Let $M \in\R^{p\times q}$ and let $m=\min (p,q)$. Then there are $U\in O(p)$, $V\in O(q)$ and (uniquely defined) \textit{singular values} $\sigma_1(M) \geq \dotsb\geq\sigma_{m}(M)\geq 0$ such that
	\begin{equation}\label{eqn:normal-form}
		M = U \begin{pmatrix} \Sigma \\ 0 \end{pmatrix} V 
		\quad\text{or}\quad
		U  \begin{pmatrix} \Sigma \; 0 \end{pmatrix} 	V, \quad \mbox{where} \quad \Sigma = \operatorname{diag}( \sigma_1(M), \dots, \sigma_m(M)),
	\end{equation}
and where $0$ is a matrix of zeroes (possibly empty). 
\end{lem}

\begin{lem}[Stability of $\sigma_k$ and $D_k$ under multiplication by orthogonal matrices] 
If $k \leq \min (p,q)$ and $M$ is a matrix in $\R^{p\times q}$,
\label{bluebird}
\begin{itemize}
\item[(i)] If $U\in O(p)$, then $D_k(UM) \sim D_k(M)$ and $\sigma_k(UM) = \sigma_k(M)$.
\item[(ii)] If $U\in O(q)$, then $D_k(MU) \sim D_k(M)$ and $\sigma_k(MU) = \sigma_k(M)$.
\end{itemize}
\end{lem}

\begin{proof} The statements $(i)$ and $(ii)$ are symmetric, so that we will only focus on $(i)$. Let \(\I\subset \{1,\dotsc, p\},\J \subset \{1,\dotsc,q\}, \#\I =\#\J=k \).
	The Cauchy-Binet identity gives
	\[
	\det \, (UM)_{i \in\I,j\in \J} = \det \, U_{i \in \I, \ell} \cdot M_{\ell, j\in \J} 
	=
	\sum_{\substack{ \mathcal{K}\subset \{1,\dotsc,p\}\\ \#\mathcal{K} = k }}
	\det \, (U_{i\ell})_{i\in\I,\ell \in \mathcal{K}}
	\det \, (M_{\ell j})_{\ell \in \mathcal{K},j\in \J}.
	\]
Hence \(D_k(UM)\lesssim D_k(M)\) and repeating the argument with \(U^{-1}M\) in place of \(M\) shows that \(D_k(M)\lesssim D_k(UM)\) as well. 

Finally, it follows from the uniqueness of the $(\sigma_i)$ in~\eqref{eqn:normal-form} that $\sigma_k(UM) = \sigma_k(M)$.
\end{proof}

\begin{cor*}[Relation between the $D_k$ and $\sigma_k$]
If $k \leq \min (p,q)$, the singular values and the maximal subdeterminants are such that
$$
\sigma_k(M) \sim D_{k-1}(M)^{-1}D_k(M),
$$
{}{}{where we use the convention that $0^{-1} 0=0$,} or equivalently
$$
D_k(M) \sim \sigma_1(M) \dots \sigma_k(M).
$$
\end{cor*}

\begin{proof}
By lemmas~\ref{SVD} and~\ref{bluebird}, it suffices to prove these formulas for a rectangular diagonal matrix; but then they are obvious.
\end{proof}

\begin{lem}\label{lem:rearrange-columns}
Given a matrix $M \in \mathbb{R}^{p\times q}$, we can change the order of its columns so that for each $k\leq\ell \leq \min (p,q)$,
	\begin{equation}\label{eqn:rearrange-columns}
		D^{(\ell)}_k(M)
		\sim
		D_k(M).
	\end{equation}
\end{lem}

\begin{proof}
We claim first that it suffices to prove the result for the matrix $UM$, where $U$ is orthogonal. Indeed, denoting $M^{(\ell)}$ for the restriction of $M$ to its first $\ell$ columns, this implies, in combination with Lemma~\ref{bluebird},
$$
D_k^{(\ell)}(M) = D_k(M^{(\ell)}) \sim D_k(UM^{(\ell)}) = D_k^{(\ell)}(UM) \sim D_k(UM) \sim D_k(M),
$$
which is the desired result. {}{}{For the remainder of the proof, we write for simplicity \(\sigma_i=\sigma_i(M)\).}

We can choose $U$ as in Lemma~\ref{SVD}, in which case, assuming for instance $p \geq q$, it suffices to deal with the case
$$
M = \left( \sigma_1 L_1, \dots \sigma_q L_q, 0, \dots 0 \right)^T,
$$
where $L_i \cdot L_j = \delta_{ij}$. The $0$ entries are irrelevant, so we can assume that
$$
M =  \left( \sigma_1 L_1, \dots \sigma_q L_q \right)^T = (M_{ij})_{1 \leq i,j \leq q}.
$$
We now claim that, after permuting the columns of $M$, it can be ensured that, for any $k$, the top left square matrix of dimension $k \times k$ has nearly maximal subdeterminant:
\begin{equation}
\label{tanager}
\det (M_{ij})_{1 \leq i,j \leq k} \sim D_k(M).
\end{equation}
 The construction of the matrix permutation is iterative and proceeds as follows: expanding the determinant of $M$ with respect to the last row, we see that
$$
\sigma_1 \dots \sigma_q = \operatorname{det} M = {}{}{\sum_{i=1}^p} (-1)^{q+i} M_{q,i} \operatorname{det} M^{\{q,i\}},
$$
where $M^{\{q,i\}}$ is the matrix obtained from $M$ by removing the $q$-th {}{}{row} and the $i$-th column. Since $|M_{q,i}| \leq \sigma_q$ and $\operatorname{det} M^{\{q,i\}} \lesssim \sigma_{1} \dots \sigma_{q-1}$ for all $i$, we can find $i_0$ such that $|M_{q,i_0}| \sim \sigma_q$, and $\operatorname{det} M^{\{q,i_0\}} \sim \sigma_{1} \dots \sigma_{q-1}$. Exchanging the columns $i_0$ and $q$, the resulting matrix satisfies~\eqref{tanager} for $k = q-1$.

We now consider the matrix $N = (M_{ij})_{1 \leq i,j \leq q-1}$, which was denoted $M^{\{q,i_0\}}$ before columns were permuted. It is such that entries in the last {}{}{row }are $\leq \sigma_{q-1}$, and subdeterminants of size $q-2$ are bounded by $\sigma_1 \dots \sigma_{q-2}$. Therefore, the same argument as above can be applied, and it proves~\eqref{tanager} for $k = q-2$. An obvious induction leads to the desired statement.
\end{proof}

\subsection{{}{}{Describing some convex bodies}}

{}{}{We can use the subdeterminants studied above to describe certain convex bodies. Our first result  concerns the measure of a neighbourhood of a convex hull.}

\begin{lem}\label{lem:det_vol} Let $X$ denote the $d \times d$ matrix with columns $x^{(i)}$. Then 
	\[
	\mes
	\left\{
	\sum_{i=1}^d t_i {x}^{(i)}+ {w}, \, |t_i|\leq 1, \, |{w}|\leq 1
	\right\}
	\lesssim
	1+ \sum_{k=1}^d D_k(X) {}{}{.}
	\]
\end{lem}

\begin{proof}
	Let $M$ be the $d \times 2d$ matrix $(X | \operatorname{Id})$, so that the set whose measure we want to estimate can be written (up to a multiplicative constant) $M B(0,1)$. By Lemma~\ref{SVD} (singular value decomposition), we can write $M = U (\Sigma | 0) V$; then
	$$
	\mes M B(0,1) = \mes U (\Sigma | 0) V B(0,1) = \mes (\Sigma | 0)  B(0,1) \lesssim \mes \Sigma B(0,1) = \det \Sigma \sim D_d(M).
	$$
	There remains to evaluate $D_d(M)$; owing to the specific structure of $M$,
	\[
	D_d(M) = D_d( X | \operatorname{Id}) \lesssim \sum_{k=1}^d D_k(X).\qedhere
	\]
\end{proof}

{}{}{We can also describe a subset of a convex hull cut out by linear inequalities, showing that it is contained in a potentially smaller convex hull.}

\begin{lem}\label{lem:basis}
	Given linearly independent \({}{}{{v}^{(1)},\dotsc{v}^{(d)}} \in (\R^d)^d,\) and \(Y_i>0,Z_i>0\) there are \( {}{}{{w}^{(1)},\dotsc{w}^{(d)}}   \in (\R^d)^d,\) with 
	\begin{equation} \label{junco}
	w^{(i)}_j \lesssim \min (Y_i|{v}^{(i)}|,Z_j)
	\end{equation}
	such that 
	\begin{equation}
	\label{oystercatcher}
		\left\{
		{z}\in\R^d: {z}=\sum_{i=1}^d y_i {v}^{(i)}, \, 
		|y_i|\leq Y_i, \,
		|z_i|\leq Z_i
		\right\}
		\subset
		\left\{
		\sum_{i=1}^d t_i {w}^{(i)}:
		|t_i| \lesssim 1
		\right\}.
	\end{equation}
\end{lem}

\begin{proof}
	
Let $Y$ be the matrix with columns $Y_i {v}^{(i)}$, which, without loss of generality, can be assumed to have nondecreasing norms. We claim that its singular values, $\tau_i$, are such that
\begin{equation}
\label{tanager0}
\tau_i \lesssim Y_i |v^{(i)}|.
\end{equation}
Indeed, by the Courant minimax principle, {}{}{$\tau_k$} can be characterized as follows
\begin{align*}
\tau_k & = \sigma_k (Y) = \min_{\dim E = d+1-k} \max_{\substack{x \in E \\ | x | =1}} | Y x| \leq \max_{\substack{x_{d-k+2} {}{}{=}\dotsb = x_{d} = 0 \\ | x | =1}} |Yx| \\
& \lesssim \max_{j = 1, \dots, d-k+1} |Y_j| |v^{(j)}| = |Y_{d-k+1}| |v^{(d-k+1)}|.
\end{align*}

Let $Z$ be the matrix $ \operatorname{diag}(Z_1,\dotsc,Z_d)$, and let {}{}{$M = ( Y^{-T} \,|\, Z^{-T} )^T \in \mathbb{R}^{2d \times d}$}. Then the set on the {}{}{left-hand side} of~\eqref{oystercatcher} is contained in $\{ z: |Mz| \lesssim 1 \}$. By Lemma~\ref{SVD}, we can write {}{}{$M = U ( \Sigma^T \,|\, 0 )^T V$}, so that the set $\{ z : |Mz| \lesssim 1 \}$ can now be written (up to a multiplicative constant)  as $W B(0,1)$, with 
\begin{equation*}
W = V^{-1} \Sigma^{-1}. 
\end{equation*}

We can now define the \({w}^{(i)}\) to be the columns of $W$; in order to establish the lemma, it suffices to prove the inequality~\eqref{junco}. Note first that 
\begin{equation} \label{tanager1}
|{w}^{(i)}|\lesssim \sigma_i(M)^{-1}. \end{equation}
Next, denoting $U = \begin{pmatrix} U_1 \; U_2 \\ U_3 \; U_4 \end{pmatrix}$ where each \(U_i\) is a \(d\times d\) matrix, we have \(Y^{-1} = U_1\Sigma V\). Therefore,
\begin{equation}
\label{tanager2}
(\tau_{d+1-i})^{-1} = \sigma_i(Y^{-1}) = \sigma_i(U_1 \Sigma V)  \lesssim \sigma_i(\Sigma) = \sigma_i(M).
\end{equation}
Combining~\eqref{tanager0}, ~\eqref{tanager1} and~\eqref{tanager2},
$$
|\vec{w}^{(i)}|\leq \sigma_i(M)^{-1} \lesssim  \tau_{d+1-i} \lesssim Y_i |v^{(i)}|.
$$

Finally, $W = Z U_3$, which gives $|w^{(i)}_j| \lesssim Z_j$.
\end{proof}

\subsection{{}{}{Extending} matrices with prescribed largest subdeterminants}

{}{}{We now start to describe the columns which may be added to a given \(p\times k\) matrix, with a prescribed effect on its singular values.}

\begin{lem}\label{lem:vol}Let $M$ be a $p \times k$ matrix, which admits a singular value decomposition as in~\eqref{eqn:normal-form}. 	{}{}{For some fixed \(C>0\), let}
$$ 
S(M,R) = {}{}{
	\{ x \in \mathbb{R}^p:D_j(M | x) \leq C D_j(M) \;\; (1 \leq j \leq \min (p,k)),\, D_{k+1}(M | x) \leq R \;\;(p \geq k+1) \} }
$$
		and set
{		\[
		\tau_i
		=
		\begin{cases}
			\sigma_i(M)&\mbox{if} \; i\leq k,
			\\
			\min (\sigma_k(M),\frac{R}{D_k(M)}) &\mbox{if} \; i\geq k+1 \; \mbox{and} \; p \geq k+1.	
		\end{cases}
		\]}
		Then, denoting $U^{(i)}$ for the columns of {}{}{the matrix $U$ from the singular value decomposition of \(M\),}
		$$
		S(M,R) \subset \left\{ \sum_{i=1}^p y_i U^{(i)} : | y_i| 	{}{}{\lesssim_{C,p,k}} \tau_i \right\}.
		$$
\end{lem}
	
\begin{proof}
	{}{}{In the proof we allow all implicit constants in \(\lesssim,\sim\) notation to depend on \(C,p,k\).}
	
	\underline{Step 1: $p \geq k+1$ and $U = \operatorname{Id}$}.  Then the singular value decomposition of $M$ is $M = \begin{pmatrix} \Sigma \\ 0 \end{pmatrix} V$ and
$$
(M | x) = \begin{pmatrix}
\Sigma V & x' \\ 0 & x'' \end{pmatrix}, \qquad  x = \begin{pmatrix} x' \\ x'' \end{pmatrix}.
$$
If $x \in S(M,R)$, it is immediate that $|x''| \lesssim \frac{R}{D_k(M)}$, by considering submatrices consisting of the first $k$ {}{}{rows,} together with one of the last $p-k$ lines). Furthermore, by considering submatrices consisting of a $(k-1) \times (k-1)$ submatrix of $\Sigma V$, one of the $p-k$ last {}{}{rows,} and the last column, we have
$$
| x''| D_{k-1}(M) \lesssim D_k(M).
$$
{}{}{It follows that \( |x''| \lesssim \sigma_k(M).\)}

We will now focus on the $k \times (k+1)$ matrix made up of the first $k$ {}{}{rows }of $(M | x)$, namely
$$
 \begin{pmatrix} \sigma_1(M) V^{(1)} &  x_1 \\ \vdots & \vdots \\ \sigma_k(M) V^{(k)}  & x_k \end{pmatrix}
$$
(where $V^{(i)}$ stands for the $i$-th {}{}{row }of $V$). We now prove by induction on $n$ that $|x_i| \leq \sigma_i(M)$ if $i \leq n$; this assertion for $n=k$ is the desired result. The case $n=1$ being immediate, we can assume the assertion holds at rank $n$, and aim at proving it at rank $n+1$.
		
		{}{}{
		The $n$ first rows of $V$ are orthogonal, therefore we can delete the last \(n-k\) rows and some \(k-n\) columns of \(V\)
		to get an  $n \times n$ matrix with a determinant $\sim 1$; denote this matrix \(\widetilde{V}\) and its rows $\widetilde{V}^{(1)},\dotsc,\widetilde{V}^{(n)} $.}
		
		{}{}{
		Note that the \(n\times n\) matrix with rows  
		$ \sigma_1(M) \widetilde{V}^{(1)}  ,\dotsc, \sigma_n(M) \widetilde{V}^{(n)} $ has determinant $\sim D_n(M)$.}
			
		{}{}{
		We now consider the submatrix of \(M\) obtained by deleting the last \(n-k-1\) rows and the same columns that were deleted from \(V\) to make \(\widetilde{V}\). That is,
		$$
		\widetilde{M} = \begin{pmatrix} \sigma_1(M) \widetilde{V}^{(1)}  & x^1 \\ \vdots & \vdots  \\ \sigma_n(M) \widetilde{V}^{(n)} & x^n \\ \sigma_{n+1}(M) \widetilde{V}^{(n+1)} & x^{n+1}  \end{pmatrix}.
		$$
		We further {}{}{write} $\widetilde{M}^{\{i , n+1 \}}$ for the matrix $\widetilde{M}$ with $i$-th row and last column removed.} Expanding the determinant of $\widetilde{M}$ with respect to the last column, we find that
$$
\det(\widetilde{M}) = \sum_{i=1}^{n+1} (-1)^{i+n+1} x^i \det \widetilde{M}^{\{i , n+1 \}} + x^{n+1} \det \widetilde{M}^{\{n+1, n+1 \}}.
$$
By the induction assumption, $|x_i| \leq \sigma_i(M)$ for $1 \leq i \leq n$. Furthermore, $\det \widetilde{M}^{\{i , n+1 \}}  \lesssim \frac{D_{n+1}(M)}{\sigma_i(M)}$ for  $1 \leq i \leq n$, and we saw that $\det \widetilde{M}^{\{n+1, n+1 \}} \sim D_n(M)$. Finally, the definition of $S(M,R)$ requires that $\det(\widetilde{M}) \lesssim D_{n+1}(M)$. Combining these observations and the above equality implies that, if $x \in S(M,R)$,
		$$
		|x^{n+1}| D_n(M) \lesssim D_{n+1}(M), \qquad \mbox{i.e.} \qquad |x^{n+1}| \lesssim \sigma_{n+1}(M).
		$$
		
		\medskip \noindent \underline{Step 2: general case $p \geq k+1$.} Then the singular value decomposition of $M$ is $M = U \begin{pmatrix} \Sigma \\ 0 \end{pmatrix} V$. Setting $y = U^{-1} x$, we can write
		$$
		D_j(UDV | x) = D_j (U (DV | y) ) \sim D_j(DV|y), \qquad \mbox{where} \qquad D = \begin{pmatrix} \Sigma \\ 0 \end{pmatrix}.
		$$
Then $x \in S(M,R)$ if and only if $y \in S(DV,R) \subset \{y: |y_i| \lesssim \tau_i \}$. The desired result follows for $x = Uy$.

\medskip \noindent \underline{Step 3: the case $p \leq k$.} Similarly to the case $p \geq k+1$, one deals first with $U = \operatorname{Id}$. Then
$$
M = (\Sigma | 0) V = (\Sigma V_1)
$$
(where $V_1$ is the {$p \times k$ upper submatrix of $V$}). Then
$$
(M|x) = ( \Sigma V_1 | x).
$$
Proceeding as in Step 1, one can deduce that $|x_i| \lesssim \tau_i$ if $1 \leq i \leq p$, and the desired conclusion follows as in Step 2.
\end{proof}

{}{}{We can apply the last lemma to Problem~\ref{countingmatricesproblem}, with some technical complexity coming from the constant entries in the last row of the matrix \(P\) appearing there. In the following lemma one should think of \(M\) as  the first \(k\) columns of \(P\), and \(x\) as being a column \((m_{1(k+1)}^2, \dotsc, m_{d(k+1)}^2, \lambda^2)^T\) to be adjoined to the matrix \(M\). As the \(m_{i(k+1)}\) range over integers of size \(\sim \mu_i\), the vector \(\mathcal{M}{x} = (m_{1(k+1)}^2\mu_1^{-1}, \dotsc, m_{d(k+1)}^2\mu_d^{-1})^T\) then takes values which are separated from each other by distances \(\gtrsim 1\). In section~\ref{cmwps} we will use this to bound the number of integral \(m_{i(k+1)}\) by the measure of a neighbourhood of the permissible real vectors \(\mathcal{M}{x}\). It is this measure that is estimated in \eqref{eqA}.}

\begin{lem}\label{lem:vol_I}
Adopting the notation of Lemma~\ref{lem:vol}, let \(\mu_1\geq \dotsc \geq \mu_{p-1} >0\) and let $\mathcal{M}$ be the $(p-1) \times p$ matrix defined by
$$
\mathcal{M}= ( \operatorname{diag}(\mu_1^{-1}, \dots , \mu_{p-1}^{-1}) | 0).
$$
As in Lemma~\ref{lem:vol} let $M$ be a $p \times k$ matrix, 	{}{}{fix \(C>0\) and put}
$$ 
S(M,R) = {}{}{
\{ x \in \mathbb{R}^p:D_j(M | x) \leq C D_j(M) \;\; (1 \leq j \leq \min (p,k)),\, D_{k+1}(M | x) \leq R \;\;(p \geq k+1) \} },
$$
and let \(
\tau_i
=
	\sigma_i(M)\) or \(
	\min (\sigma_k(M),\frac{R}{D_k(M)})\) for \(i\leq k\) or \(i>k\) respectively. Then, for any $A>0$, if $M_{p1}>\epsilon \sigma_{1}(M)$ for some $\epsilon >0$, then
\begin{equation}
\label{eqA}
\mes \{\mathcal{M}{x}+ {w}:{x}\in S(M,R), \, x_p=A,\,{}{}{|x_i| \in [\tfrac{\mu_i^2}{C}, C \mu_i^2]} \;(i<p) ,\, |{w}|\leq 1\}
{}{}{\lesssim_{\epsilon,p,k,C}} 1 + \sum_{k=1}^{p-1} D_k(\widetilde{W}),
\end{equation}
where $\widetilde{W}$ is a $(p-1) \times (p-1)$ matrix with entries such that
$$
| \widetilde{W}_{ij} | {}{}{\lesssim_{p,k,C}} \mu_j^{-1} \min (\tau_{i+1},\mu_j^2),
$$
so that \(D_k(\widetilde{W}) {}{}{\lesssim_{p,k,C}} \max_{\substack{ i_1,\dotsc,i_k\textup{ distinct}\\j_1,\dotsc,j_k\textup{ distinct} }}
\prod_{\ell=1}^{k}
\mu_{j_\ell}^{-1} \min (\tau_{i_\ell+1},\mu_{j_\ell}^2)\).
\end{lem}

\begin{proof} 
	{}{}{In the proof we allow all implicit constants in \(\lesssim,\sim\) notation to depend on \(C,p,k\).} {Taking the difference of two vectors in the set on the right-hand side of~\eqref{eqA}, we see that it suffices to prove the desired statement for $A=0$, and the condition $|x_i| \sim \mu_i^2$ replaced by $|x_i| \lesssim \mu_i^2$. In other words, it suffices to prove 
$$
\mes \{\mathcal{M}{x}+ {w} : {x}\in S(M,R), \, x_p=0,\, |x_i| \lesssim \mu_i^2 \;(i<p) ,\, |{w}|\leq 1\}
\lesssim_\epsilon 1 + \sum_{k=1}^{p-1} D_k(\widetilde{W}).
$$}
Define the projector $P$ on the first $p-1$ coordinates of a vector of $\mathbb{R}^p$:
$$
P( (x_1 ,\dots, x_p)^T) = (x_1, \dots, x_{p-1}, 0)^T.
$$
Let \(U, V\) be matrices as in \eqref{eqn:normal-form} and let \(U^{(i)}\) be the \(i\)th column of \(U\). Since $M_{p1}>\epsilon \sigma_{1}(M)$ there is \(i_0\) such that \([MV^{-1}]_{pi_0}\gtrsim \epsilon\sigma_{1}(M)\), that is to say \(U^{(i_0)}_p\sigma_{i_0}(M)\gtrsim \epsilon\sigma_{1}(M)\), whence \(U^{(i_0)}_p\gtrsim_\epsilon 1\) and \(\sigma_{i_0}(M)\gtrsim_\epsilon \sigma_{1}(M)\).

By Lemma~\ref{lem:vol},
$$
\left\{ x \in S(M,R) : x_p=0 \right\} \subset \left\{x:x=\sum_{i=1}^p y_i U^{(i)} : |y_i| \lesssim \tau_i, \, x_p=0 \right\}.
$$
Since the $p$-th coordinate of $\sum_{i=1}^p y_i U^{(i)}$ is $0$, we find that
$$
P \left[ \sum_{i=1}^p y_i U^{(i)} \right] = \sum_{i =1}^{p-1}  \widetilde{ y_i} \widetilde{U}^{(i)},
$$
where
\begin{align*}
 \widetilde{U}^{(i)} &= P \left[ U^{(i)} - \frac{U^{(i)}_p}{U^{(i_0)}_p} U^{(i_0)} \right],
 &  \widetilde{ y_i} &= y_{i}
&(i<i_0),
\\
 \widetilde{U}^{(i)} &= P \left[ U^{(i+1)} - \frac{U^{(i+1)}_p}{U^{(i_0)}_p} U^{(i_0)} \right],
&  \widetilde{ y_i} &= y_{i+1}
&(i\geq i_0),
\end{align*}
and our choice of \(i_0\) above ensures that 
 \(| \widetilde{U}^{(i)} |\lesssim_\epsilon 1\). 
Therefore,
$$
\{ P x : x \in S(M,R), \, x_p=0 \} \subset \left\{x:x= \sum_{i=1}^{p-1} \widetilde{y_i} \widetilde{U}^{(i)}, \, |\widetilde{y_i}| \lesssim \tau_{i+1} \right\}.
$$
We now add the condition  $|x_i| \lesssim \mu_i^2$ to obtain
\begin{align*}
& \{ Px : x\in S(M,R), \, |x_i| \lesssim \mu_i^2 \, \,\mbox{if $1 \leq i \leq p-1$}, \, x_p =0\} \subset 
\left\{x: x= \sum_{i=1}^{p-1} \widetilde{y_i} \widetilde{U}^{(i)}, \, |\widetilde{y_i}| \lesssim \tau_{i+1},\, |x_i| \lesssim \mu_i^2 \right\} \\
& \qquad \qquad \qquad \subset \left\{ \sum_{t=1}^{p-1} t_i w^{(i)}: |t_i| \leq 1 \right\} \qquad \mbox{with} \;\;\; |w^{(i)}_j| \lesssim \min (\tau_{i+1}, \mu_j^2),
\end{align*}
where the last inclusion is a consequence of Lemma~\ref{lem:basis}. Applying the matrix $\mathcal{M}$, we see that
\begin{align*}
& \{ \mathcal{M}x:x\in S(M,R), \, |x_i| \lesssim \mu_i^2 \,\, \mbox{if $1 \leq i \leq p-1$}, \, x_p =0\} \\
& \qquad \qquad  \subset \left\{ \sum_{t=1}^{p-1} t_i \widetilde{w}^{(i)} : |t_i| \leq 1 \right\}  \qquad \mbox{with} \;\;\; |\widetilde{w}^{(i)}_j| = | [\mathcal{M} w^{(i)}]_j |\lesssim \mu_j^{-1} \min (\tau_{i+1}, \mu_j^2).
\end{align*}
Finally, Lemma~\ref{lem:det_vol}  gives the desired conclusion.
\end{proof}

\section{Proof of Theorem~\ref{thmgeometryofnumbers}}\label{sec:geomofnumthm}

In this section we will prove the following result.

\begin{thm}\label{thm:generic-L-infty}
	The following holds for any fixed off-diagonal coefficients \(\beta_{ij}=\beta_{ji}\in [-\frac{1}{10d^2},\frac{1}{10d^2}]\) \((i<j)\) and   
almost all $(\beta_{11},\dotsc,\beta_{dd})^T \in [1,2]^d$. Moreover it also holds for almost all symmetric matrices with \(\beta_{ji}\in [-\frac{1}{10d^2},\frac{1}{10d^2}]\) for \(i\neq j\) and \(\beta_{11},\dotsc,\beta_{dd}\in[1,2]\).

For every \(b\in\N\), every \(\delta <1<\lambda\) and  any $\epsilon>0$, we have
\begin{equation}\label{eq:thm-sharp}
	\lVert P_{\lambda,\delta} \rVert_{L^1\to L^\infty} \lesssim_{\beta,\epsilon} \delta^{-1/b} \lambda^{\epsilon} \left[ \max_{0 \leq b_2 \leq \min (b,d+1)}  \delta^{ b_2} \lambda^{-b_2^2 + (b+d) b_2 + {1-b}} \right]^{1/b}.
\end{equation}

It follows in particular that, for every $a\in\N$ and whenever \(\lambda^{-a}\leq \delta\leq \lambda^{1-a}\), we have
\begin{equation}\label{eq:thm-blunt}
\lVert P_{\lambda,\delta} \rVert_{L^1\to L^\infty}  \lesssim_{\beta,\epsilon} 
\begin{cases}
\delta^{1- \frac 1 {d+1-a}} \lambda^{d - 1 + {\frac{1}{d+1-a}} + \epsilon}
&(a<d
),
\\
{\delta^{1-\frac {1+a} {d+1+a}} \lambda^{ d -1+\frac{1+a}{d+1+a}+ \epsilon}}
&( a\geq d).
\end{cases}
\end{equation}
\end{thm}

We emphasize that \eqref{eq:thm-blunt} is intended as illustrative; one could prove a stronger but less tidy result just by making a more careful choice of the parameter \(b\) at the end of the proof.

By contrast it seems more challenging to improve \eqref{eq:thm-sharp} using our methods. See Remark~\ref{rem:not-optimal} for one idea.

Before proceeding to the proof we deduce  Theorem~\ref{thmgeometryofnumbers}.

\begin{proof}[Proof of Theorem~\ref{thmgeometryofnumbers}]
	Equations \eqref{eq:thm_blunt_main} and \eqref{eq:thm_blunt_edge} are cases of \eqref{eq:thm-blunt}. For \eqref{eq:thm_delta_very_small}, observe that if
	\(
	\|P_{\lambda,\delta}\|
	\)
	is sufficiently large in terms of the function \(\chi\), then there are two integer vectors with
	\[
	{\mathcal{Q}(x^{(1)})}^{1/2}, {\mathcal{Q}(x^{(2) })}^{1/2}\in [\lambda-\delta,\lambda+\delta],
	\quad\text{and}\quad
	|x^{(1)}_i|\neq |x^{(2)}_i|
	\quad\text{ for some \(i\).}
	\]
	Letting \(y_{ij}=x^{(1)}_ix^{(1)}_j-x^{(2)}_ix^{(2)}_j\) we find that
	\begin{align}	
		y_{ij}&\in\Z,
		&|y_{ij}|&\lesssim \lambda^2,
		&y_{ii}&\neq 0\text{ for some }i,
		&\Big\lvert
		\sum_{i.j} \beta_{ij}y_{ij}
		\Big\rvert&\lesssim \delta\lambda.
		\label{eq:mat_y}
	\end{align}
	 For a given matrix \(y_{ij}\) and for any off-diagonal coefficients \(\beta_{ij}=\beta_{ji}\) \((i<j)\) we have
	\[
	\mes\Big\{
	(\beta_{11},\dotsc,\beta_{dd})^T\in[1,2]^d :
	\Big\lvert
	\sum_{i.j} \beta_{ij}y_{ij}
	\Big\rvert\leq \delta\lambda
	\Big\}
	\lesssim
	\frac{\delta\lambda}{\max\{|y_{ij}|\} }.
	\]
	We claim that
	\[
	\#\{(y_{ij}) : y_{ij}=x^{(1)}_ix^{(1)}_j-x^{(2)}_ix^{(2)}_j\text{ for some }x^{(i)}\in \Z^d,\text{ and }
	0<\max\{|y_{ij}|\}\leq Y \}
	\lesssim_\epsilon Y^{d+\epsilon}.
	\]
	For the proof, note that we cannot have every \(y_{ii}=0\) or else \((y_{ij})\) would vanish. We assume without loss of generality that there is \(i_0\in\{1,\dotsc,d\}\) such that \(y_{ii}=0\) iff \(i> i_0\). There are \(O(Y^{i_0})\) possible values of  \(y_{ii}\) for \(i\leq i_0\), and once these are chosen the identity \(y_{ii} =(x^{(1)}_i+x^{(2)}_i)(x^{(1)}_i-x^{(2)}_i) \) determines  \(x^{(1)}_i,x^{(2)}_i\) for \(i\leq i_0\) up to \(Y^\epsilon\) possibilities.  Next for \(i>i_0\) we have  \(x^{(1)}_i=\pm x^{(2)}_i\), and so up to finitely many choices both of these are determined  by the values of \(y_{1i} = (\pm x^{(1)}_1\pm x^{(2)}_1)x^{(1)}_i\), for which there are \(O(Y^{d-i_0})\) possibilities. We conclude that there are \(\lesssim_\epsilon Y^{d+\epsilon}\) choices for the \(x^{(i)}\) and hence for \((y_{ij})\).
	
	We can now conclude that, for a suitably large  constant \(C\) depending only on the cutoff function \(\chi\), and for any fixed values of the off-diagonal entries \(|\beta_{ij}|\leq \frac{1}{10d^2}\) (\(i< j\)), we have
	\begin{align*}
	\mes \{ 
	(\beta_{11},\dotsc,\beta_{dd})^T\in[1,2]^d :
	\|P_{\lambda,\delta}\|>C\text{ for some } \lambda\sim \lambda_0, \delta < \delta_0
	\}
	&\lesssim
	\sum_{\substack{(y_{ij})\\ 0<\max\{|y_{ij}|\}\lesssim \lambda_0^2 }}
	\frac{\delta_0\lambda_0}{\max\{|y_{ij}|\} }
	\\&
	\lesssim_\epsilon
	\sum_{\substack{Y=2^k \\ k\in \Z \\ Y\lesssim \lambda_0^2}}
	\delta_0\lambda_0Y^{d-1+\epsilon}
	.
	\end{align*}
	This is \(O_\epsilon(\delta_0\lambda_0^{2d-1})\). Applying the Borel-Cantelli lemma (Lemma~\ref{borelcantelli}) proves \eqref{eq:thm_delta_very_small}. 
\end{proof}

We now begin the proof of Theorem~\ref{thm:generic-L-infty}. Throughout the rest of this section we write \(\beta_i\) for \(\beta_{ii}\), put \(\beta'=(\beta_1,\dotsc \beta_d)^T\), and 
given $b,d\in \mathbb{N}$ and $M = (m_{ij})_{1\leq i\leq d,\,1\leq j\leq b}$, we put
\begin{equation*}
	P(M)
	=
	(p_{ij}(M))_{1\leq i \leq d+1,\,1\leq j\leq b}
	=
	\left(
	\begin{matrix}
		m^2_{11}&\cdots&m^2_{1b}\\
		\vdots&\ddots&\vdots\\
		m^2_{d1}&\cdots&m^2_{db}\\
		\lambda_0^2&\cdots&\lambda_0^2
	\end{matrix}
	\right).
\end{equation*}

\subsection{Integrating over $\lambda$ and $\beta$}

Our key observation is as follows. {}{}{Since, for \(m\in\mathbb{Z}\), we have \(1\leq \sum_{\mu \in 2^{\mathbb{N}}\cup \{0\}} \mathbf{1}_{\mu\leq 2m\leq 2\mu},\) and since \(\chi\) takes non-negative values, we have for any $\lambda,\delta>0$ that}
%
%
\begin{align}
	\MoveEqLeft[4]
	\int_{[1,2]^d}
	\int_{\lambda_0/2}^{\lambda_0}
	\lVert P_{\lambda,\delta} \rVert_{L^1\to L^\infty}^b
	\, d \lambda
	\, d {\beta'}
	\nonumber
	\\
	&=
	\int_{[1,2]^d}
	\int_{\lambda_0/2}^{\lambda_0}
	\bigg(
	\sum_{\substack{\lambda_0 \geq\mu_1,\dotsc, \mu_d \geq 0\\ \mu_i \in 2^{\mathbb{N}}\cup \{0\} }}
	\sum_{\substack{{m} \in \mathbb{Z}^d \\ |m_i|\in [\frac{\mu_i}{2},\mu_i] }}
	\chi \left( \frac{\mathcal{Q}({m}) - \lambda^2}{\delta \lambda} \right)
	\bigg)^b
	\, d \lambda
	\, d {\beta'}
	\nonumber
	\\
	&\lesssim
	\int_{[1,2]^d}
	\int_{\lambda_0/2}^{\lambda_0}
	\log^{bd}(\lambda_0)
	\max_{\substack{\lambda_0 \geq\mu_1\geq\dotsb\geq \mu_d \geq 0\\ \mu_i \in 2^{\mathbb{N}}\cup \{0\} }}
	\bigg(
	\sum_{\substack{{m} \in \mathbb{Z}^d \\  |m_i|\in [\frac{\mu_i}{2},\mu_i] }}
	\chi \left( \frac{\mathcal{Q}({m}) - \lambda^2}{\delta \lambda} \right)
	\bigg)^b
	\, d \lambda
	\, d {\beta'},
	\nonumber
	\\
	&
	{}{}{\lesssim
	\int_{[1,2]^d}
	\int_{\lambda_0/2}^{\lambda_0}
	\log^{bd}(\lambda_0)
	\max_{\substack{\lambda_0 \geq\mu_1\geq\dotsb\geq \mu_d \geq 0\\ \mu_i \in 2^{\mathbb{N}}\cup \{0\} }}
	\prod_{j=1}^b
	\sum_{\substack{m_{1j},\dotsc,m_{dj}\in \mathbb{Z} \\  |m_{ij}|\in [\frac{\mu_i}{2},\mu_i] }}
	\mathbf{1}_{ \lvert {\mathcal{Q}(m_{1j},\dotsc,m_{dj}) - \lambda^2}\rvert \leq \delta \lambda} 
	\, d \lambda
	\, d {\beta'},}
	\nonumber
\end{align}
{}{}{
and if we temporarily write the off-diagonal parts of \(\mathcal{Q}(m_{1j},\dotsc,m_{dj})\) using the row vector
	\begin{equation*}
		q
		=
		\Big(\sum_{\substack{1\leq i,j\leq d \\ i \neq j}} \beta_{ij} m_{i1}m_{j1},\dotsc, \sum_{\substack{1\leq i,j\leq d \\ i \neq j}}  \beta_{ij} m_{ib}m_{jb}\Big).
	\end{equation*}
	then this becomes
	\begin{align}
		\MoveEqLeft[1]
		\int_{[1,2]^d}
		\int_{\lambda_0/2}^{\lambda_0}
		\lVert P_{\lambda,\delta} \rVert_{L^1\to L^\infty}^b
		\, d \lambda
		\, d {\beta'}\nonumber
		\\
	&\lesssim
		\int_{[1,2]^d}
		\int_{\lambda_0/2}^{\lambda_0}
		\log^{bd}(\lambda_0)
		\max_{\substack{\lambda_0 \geq\mu_1\geq\dotsb\geq \mu_d \geq 0\\ \mu_i \in 2^{\mathbb{N}}\cup \{0\} }}
		\prod_{j=1}^b
		\sum_{\substack{m_{1j},\dotsc,m_{dj}\in \mathbb{Z} \\  |m_{ij}|\in [\frac{\mu_i}{2},\mu_i] }}
		\mathbf{1}_{ \lvert(\beta_1,\dotsc,\beta_d, - \lambda^2/\lambda_0^2) P(M)-q\rvert \leq \delta \lambda_0}
		\, d \lambda
		\, d {\beta'}\nonumber
	\\
	&=
	\log^{bd}(\lambda_0)
	\max_{\substack{\lambda_0 \geq\mu_1\geq\dotsb\geq \mu_d \geq 0\\ \mu_i \in 2^{\mathbb{N}}\cup \{0\} }}\nonumber
	\\
	&\hphantom{{}={}}
	\sum_{\substack{M\in \mathbb{Z}^{d\times b} \\ |m_{ij}| \in [\mu_i/2,\mu_i] }}
	\mes\big\{
	({\beta'},\lambda)\in [1,2]^d\times [\frac{\lambda_0}{2},\lambda_0] : \lvert(\beta_1,\dotsc,\beta_d, - \lambda^2/\lambda_0^2) P(M)-q\rvert \leq \delta \lambda_0
	\big\},
	\label{eqn:large-measure}
\end{align}
We can estimate the measure inside the last}
 sum  in \eqref{eqn:large-measure} as follows.
Notice first that
\begin{equation}
	\label{eqn:largest-singular-value}
	D_1(P(M)) = \max_{\substack{i=1,\dotsc, d+1 \\ j=1,\dotsc,b}} \lvert p_{ij}(M) \rvert \sim \lambda_0^2.
\end{equation}
By Lemma~\ref{SVD}{}{}{, we have
\begin{multline*}
	\mes\Big\{
	{\gamma}\in \R^{d+1} : \lvert {\gamma}\rvert \leq 1, \,
	\lvert 
	\sum_i \gamma_ip_{ik}(M)
	-
	\sum_{i \neq j} \beta_{ij} m_{ik}m_{jk}
	\rvert \leq \delta \lambda_0 
	\Big\}
	\\
	{}{}{\lesssim}
	\prod_{k=1}^{\min (b,d+1)}
	\min\left( 1,
	\frac{\delta \lambda_0}{\sigma_k(P(M))} \right) =  \prod_{\substack{1 \leq i \leq \min (b,d+1) \\ \sigma_i(P(M))>\delta \lambda_0}} \frac{\delta \lambda_0}{\sigma_i(P(M))}
.
\end{multline*}}
Together with \eqref{eqn:large-measure}  {}{}{and the fact that \(P(M)\) does not depend on the signs of the \(m_{ij}\),} we find
\begin{multline*}
	\int_{[1,2]^d}
	\int_{\lambda_0/2}^{\lambda_0}
	\lVert P_{\lambda,\delta} \rVert_{L^1\to L^\infty}^b
	\, d \lambda
	\, d {\beta'}
	\lesssim
	(\log\lambda_0)^{bd+d}
	\max_{\substack{\lambda_0 \geq\mu_1\geq\dotsb\geq \mu_d \geq 0\\ \mu_i \in 2^{\mathbb{N}}\cup \{0\} }}
	\sum_{\substack{M\in \mathbb{Z}^{d\times b} \\ m_{ij} \in [\mu_i/2,\mu_i] }}
	\lambda_0
 \prod_{\substack{1 \leq i \leq \min (b,d+1) \\ \sigma_i(P(M))>\delta \lambda_0}} \frac{\delta \lambda_0}{\sigma_i(P(M))},
\end{multline*}
{}{}{where the \(m_{ij}\) are now non-negative since \(m_{ij} \in [\mu_i/2,\mu_i] \).}
Combining this with \eqref{eqn:largest-singular-value} and {}{}{the Corollary to Lemma~\ref{bluebird} yields
\begin{equation}
	\int_{[1,2]^d}
	\int_{\lambda_0/2}^{\lambda_0}
	\lVert P_{\lambda,\delta} \rVert_{L^1\to L^\infty}^b
	\, d \lambda
	\, d {\beta'}
	\lesssim
	(\log\lambda_0)^{bd+d}
	\max_{\substack{ \lambda_0^2 \sim L_1\geq \dotsb\geq L_{{\min (b,d+1)}} \geq 0 \\ \lambda_0 \geq\mu_1\geq\dotsb\geq \mu_d \geq 0
\\\mu_i \in 2^{\mathbb{N}}\cup \{0\} 	
}}
	\lambda_0
	\prod_{\substack{1 \leq i \leq \min (b,d+1) \\ L_i>\delta \lambda_0}} \frac{\delta \lambda_0}{L_i}
	Z_{d,b}(\vec{\mu},\vec{L})
	,
	\label{eqn:prescribed-subdeterminants}
\end{equation}
where
\begin{multline}\label{eqn:matrix_count}
Z_{d,b}(\vec{\mu},\vec{L})=
	\#\Big\{ M\in \mathbb{Z}^{d\times b} :  
	\frac{m_{ij}}{\mu_i }\in [\tfrac{1}{2},1],\,
	 \frac{ D_k(P(M))}{L_1\dotsm L_k }\in [\tfrac{1}{2},1]
	 \\\text{for all } 1\leq i \leq d,\;1\leq j\leq b,\;1\leq k \leq {\min (b,d+1)} \Big\}.
\end{multline}
}{}{}{In \eqref{eqn:prescribed-subdeterminants} we may assume that \((\mu_i=0\implies L_{i+1}=0)\), since otherwise \(
Z_{d,b}(\vec{\mu},\vec{L})\) would be zero (there are no such \(M\)).}
In particular allowing \(\mu_i\) to be zero is the same as allowing the dimension \(d\) to drop, in the sense that{}{}{
\begin{multline*}
	\max_{\substack{ \lambda_0^2 \sim L_1\geq \dotsb\geq L_{{\min (b,d+1)}} \geq 0 \\ \lambda_0 \geq\mu_1\geq\dotsb\geq \mu_d \geq 0\\\mu_i \in 2^{\mathbb{N}}\cup \{0\} 
	}}
	\lambda_0
\left[	\prod_{\substack{1 \leq i \leq \min (b,d+1) \\ L_i>\delta \lambda_0}} \frac{\delta \lambda_0}{L_i} \right]
Z_{d,b}(\vec{\mu},\vec{L})
	\\
	=
	\max_{\substack{ 0\leq d'\leq d\\\lambda_0^2 \sim L_1\geq \dotsb\geq L_{{\min (b,d'+1)}} \geq 0 \\ \lambda_0 \geq\mu_1\geq\dotsb\geq \mu_{d'} \geq 1,\,\mu_i \in 2^{\mathbb{N}} }}
	\lambda_0
\left[ \prod_{\substack{1 \leq i \leq \min (b,d'+1) \\ L_i>\delta \lambda_0}} \frac{\delta \lambda_0}{L_i} \right]
Z_{d',b}(\vec{\mu},\vec{L}).
\end{multline*}}

\subsection{Counting matrices with prescribed subdeterminants}

\label{cmwps}


We want to estimate the right-hand side of \eqref{eqn:prescribed-subdeterminants}, \underline{under the assumption that $\mu_i \neq 0$ for any $i$}, or in other words $\mu_i \in 2^\mathbb{N}$. 

{}{}{Our first object is to estimate from above the number of matrices \(M\) counted by the function \(Z_{d,b}(\vec{\mu},\vec{L})\) from \eqref{eqn:matrix_count}. By Lemma~\ref{lem:rearrange-columns}, it suffices to l count  those \(M\) satisfying an additional condition
	\[
	D_k(P(M))\sim D_k^{(k)}(P(M))
	{}{}{\text{ for all }} \leq k \leq {\min (b,d+1)},
	\]
	since permuting the columns of these recovers all the matrices in \(Z_{d,b}(\vec{\mu},\vec{L})\) }

For $j=1,\dotsc,b$ define the vectors $m^{(j)}$ and ${n}^{(j)}\in\R^d$ {}{}{by}
$$
m^{(j)} = {}{}{(m_{1j},\dotsc,m_{dj})^T,} \qquad 
{n}^{(j)}
=
(m_{1j}^2/\mu_1,\dotsc,m_{dj}^2/\mu_d)^T.
$$
{}{}{That is the vectors $m^{(j)}$ are the columns of $M$. Meanwhile the vectors $n^{(j)}$ are the columns of $P(M)$ with the last element dropped and the others rescaled to that   $n^{(j)}$ belongs to the set $S$ defined by
	$$
S
=
\{
(u_{1}^2/\mu_1,\dotsc,u_{d}^2/\mu_d)^T: u_i\in \mathbb{Z},\, u_i \in [\frac{\mu_i}{2},\mu_i]  \},
$$}
{}{}{whose elements are separated by gaps of size \(\sim 1\)}.
If $M$ is counted in the right-hand side of {}{}{\eqref{eqn:matrix_count}}, then the vector ${n}^{(1)}$ can be chosen arbitrarily from $S$; there are ${}{}{\lesssim} \prod_{i=1}^d \mu_i$ choices.

Suppose now that the first $k$ columns of \(P(M)\) are given, and they satisfy
		\[
		 D_\ell(P(m^{(1)}|\dotsb| m^{(k)}))
		 \sim L_1\dotsm L_\ell
		 \qquad
		 (1\leq \ell \leq \min ( k,d+1)).
		\]
		 {}{}{We} want to select $m^{(k+1)}$, or equivalently $n^{(k+1)}$. We can first use that $S$ is 1-separated to replace our counting problem by {}{}{a} volume estimate: letting
\begin{align*}
\mathcal{N}^{k+1} = & \{ n^{(k+1)} \in S:
{}{}{
m^{(k)}_i
 \in [\mu_i/2,\mu_i],}
\, D_\ell(P(m^{(1)},\dots m^{(k+1)})) \sim L_1 \dots L_\ell \\
& \qquad \qquad \qquad \qquad \qquad \qquad \qquad \,\,(1\leq i \leq d, \,1\leq \ell \leq \min (k+1,d+1)) \},
\end{align*}
then
$$
\# \mathcal{N}^{k+1}  \lesssim \mes \left[ \mathcal{N}^{k+1} + B(0,1) \right].
$$
\begin{rem}\label{rem:not-optimal}
	This volume bound is not necessarily optimal. To take just one simple example, if \(\lambda_0^2\) is an integer then \(D_\ell^{(\ell)}(P(M))\) is an integer. Thus, in the case when \(0<L_1\dotsm L_\ell \ll 1\), it is impossible for \(D_\ell^{(\ell)}(P(M))\sim L_1\dotsm L_\ell \) to hold and the set \(\mathcal{N}^{k+1}\) is empty.
\end{rem}
We apply  Lemma~\ref{lem:vol_I}  with
\begin{align*}
	p&=d+1,
	&
	M&={}{}{P(m^{(1)}|\dots |m^{(k)}),}
	&
	\mathcal{M}x&=n^{(k+1)},
	&
	A&=\lambda_0^2,
	&
	R=L_1\dotsm L_{k+1}.
\end{align*}
We compute that
\[
\tau_i
\asymp
\begin{cases}
\max\{
L_i,L_{k+1}
\}
&(k\leq d),
\\
L_i
&(k>d).
\end{cases}
\]
We now need to distinguish two cases. If \(k\leq d\) then  applying Lemma~\ref{lem:vol_I} gives that 	
\[
\# \mathcal{N}^{k+1}
	\lesssim
	1+
	\max_{\substack{ \I\subset \{1,\dotsc,d\} \\ \J \subset \{1,\dotsc,d\} \\ \#\I =\#\J\\ \sigma: \mathcal{I}\hookrightarrow \mathcal{J} }}
	\bigg(\prod_{i\in\mathcal{I}} \mu_i \bigg)
	\bigg( \prod_{\substack{\mu_i^2 > \max(L_{\sigma(i)+1},L_{k+1})\\i\in\mathcal{I}}} \frac{\max(L_{\sigma(i)+1},L_{k+1})}{\mu_i^2} \bigg).
	\]
	 If instead {$k >d$}, applying Lemma~\ref{lem:vol_I} gives that
		\[
		\# \mathcal{N}^{k+1}
	\lesssim
	1+
	\max_{\substack{ \I\subset \{1,\dotsc, d\} \\ \J \subset \{1,\dotsc,d\} \\ \#\I =\#\J\\ \sigma: \mathcal{I}\hookrightarrow \mathcal{J} }}
	\bigg(\prod_{i\in\mathcal{I}} \mu_i \bigg)
	\bigg( \prod_{\substack{{\mu_i^2 > L_{\sigma(i)+1}} \\i\in\mathcal{I}}} \frac{ L_{\sigma(i)+1}}{\mu_i^2} \bigg).
	\]
Recall now that this is a bound for the number of choices for \(m^{(k+1)}\), given \(m^{(1)},\dotsc,m^{(k)}\), and that there are \(\mu_1\dotsm \mu_d\) choices for \(m^{(1)}\). Recall also that our object is to estimate that part of the right-hand side of \eqref{eqn:prescribed-subdeterminants} for which every \(\mu_i\) is nonzero. The bound we have proved is
\begin{multline}
\label{formulamax}
\max_{\substack{ \lambda_0^2 \sim L_1\geq \dotsb\geq L_{{\min (b,d+1)}} \geq 0 \\ \lambda_0 \geq\mu_1\geq\dotsb\geq \mu_d > 0
		\\\mu_i \in 2^{\mathbb{N}}
}}
\lambda_0
\prod_{\substack{1 \leq i \leq \min (b,d+1) \\ L_i>\delta \lambda_0}} \frac{\delta \lambda_0}{L_i} 
Z_{d,b}(\vec{\mu},\vec{L})
\\
\lesssim
	\max_{ \substack{\lambda_0^2 \sim L_1\geq \dotsb\geq L_{\min (b,d+1)} \geq 0 \\ \lambda_0 \geq\mu_1\geq\dotsb\geq \mu_d \geq 0}}
	\max_{\substack{ \I_k \subset \{1,\dotsc,d\} \\ \sigma_k: \mathcal{I}_k \hookrightarrow \{1,\dots, d \} }}
	\mathcal{F}((L_i),(\mu_i),(\mathcal{I}_k),(\sigma_k))
\end{multline}
where $\sigma_k$ is a bijection from $\mathcal{I}_k$ to a subset of $\{1,\dots,d\}$ and
\begin{align*}
& \mathcal{F}((L_i),(\mu_i),(\mathcal{I}_k),(\sigma_k)) \\
& = \lambda_0
	\bigg(
	\prod_{\substack{i=1,\dotsc,\min (b,d+1) \\ L_i>\delta \lambda_0}} \frac{\delta \lambda_0}{L_i}\bigg)
	\bigg(\prod_{i=1}^d \mu_i \bigg)
	\prod_{k=1}^{\min (b-1,d)}	
	\bigg(\prod_{i\in\mathcal{I}_k} \mu_i \bigg)
	\bigg( \prod_{\substack{
			\mu_i^2 > \max\{L_{\sigma_k(i)+1},L_{k+1}\}\\i\in\mathcal{I}_k}}
		\frac{\max(L_{\sigma_k(i)+1},L_{k+1})
	}{\mu_i^2} \bigg) \\
	&  \qquad \qquad \qquad \qquad \prod_{k=\min (b,d+1)}^{b-1}	
	\bigg(\prod_{i\in\mathcal{I}_k} \mu_i \bigg)
	\bigg( \prod_{\substack{\mu_i^2 >L_{\sigma_k(i)+1}\\i\in\mathcal{I}_k }}
	\frac{L_{\sigma_k(i)+1}}{\mu_i^2} \bigg)
\end{align*}
with the understanding that $\mathcal{I}$ might be empty and $\prod_{\mathcal{I}} = 1$ if $\mathcal{I} = \emptyset$.

\subsection{The maximization procedure}
Our aim is now to find the values of \(L_i,\mu_i,\mathcal{I}_k,\sigma_k\) for which the maximum on the right-hand side of~\eqref{formulamax} is attained. 

\medskip
\noindent {}{}{\underline{Step 1: Maximizing in $(L_i)$ with the $\mathcal{I}_k$'s and $\mu_i$'s held fixed.}} We start with the dependence on $(L_i)$, and we relax first the condition that they be ordered; we will simply assume that {}{}{$0 \leq L_i \leq \lambda_0^2$ for each $i=1, \dotsc, \min (b,d+1)$}. Next, we claim that the {}{}{products}
$$
\bigg( \prod_{\substack{
			\mu_i^2 > \max (L_{\sigma_k(i)+1},L_{k+1}) \\i\in\mathcal{I}_k}}
		\frac{\max(L_{\sigma_k(i)+1},L_{k+1})
	}{\mu_i^2} \bigg) \quad \mbox{and} \quad
	\bigg( \prod_{\substack{\mu_i^2 >L_{\sigma_k(i)+1}\\i\in\mathcal{I}_k }}
	\frac{L_{\sigma_k(i)+1}}{\mu_i^2} \bigg)
$$
 on the right-hand side of the definition of $\mathcal{F}((L_i),(\mu_i),(\mathcal{I}_k),(\sigma_k))$ can be taken to be empty. Indeed, assume that the maximum of $\mathcal{F}$ is reached at a point such that, for some $k$ in the product, the corresponding {}{}{product} is not empty: it contains $\frac{\max(L_{\sigma_k(i)+1},L_{k+1})}{\mu_i^2} $ or $\frac{L_{\sigma_k(i)+1}}{\mu_i^2}$ for some $i$; we will denote it $\frac{L_{i_0}}{\mu_{i_1}^2}$. Expand $\mathcal{F}$ in powers of the $L_i$, and consider the exponent of $L_{i_0}$. It is necessarily $\geq 0$, since only the first factor in the definition of $\mathcal{F}$ contributes a negative power, namely $L_{i_0}^{-1}$. Therefore, $\mathcal{F}$ will be larger, or equal, if we increase the value of $\max(L_{\sigma_k(i)+1},L_{k})$ to $\mu_{i_1}^2$, which has the effect of {}{}{cancelling} the undesirable term.

{After this manipulation, the parentheses we mentioned have been cancelled, and the value of some of the $L_i$'s has been fixed to $\mu_{f(i)}^2$ for some function $f$. The remaining $L_i$ contribute $\mathbf{1}_{L_i > \delta \lambda_0} \frac{\delta \lambda_0}{L_i }$, and they might be constrained by inequalities of the type $L_i > \mu_j^2$. Therefore, $\mathcal{F}$ will be maximal if they take the value $\delta \lambda_0$, or $\mu_{f(i)}^2$ for some function $f$. }

\medskip
\noindent {}{}{\underline{Step 2: Maximizing in $(\mu_i)$.}} The result of the maximization in $(L_i)$ is that we can assume that each $L_i$ takes the value either $\mu_{f(i)}^2$, or $\delta \lambda_0$, that $\mu_i^2 \leq \max(L_{\sigma_k(i)+1},L_{k+1}) \; \mbox{if $i \in \mathcal{I}_k$}$, with the convention that $L_{k} = 0$ for {$k \geq d+2$}, and that the function to maximize is
\begin{equation}
\label{functiontomaximize}
\lambda_0
	\bigg(
	\prod_{\substack{i=1,\dotsc,\min (b,d+1) \\ L_i>\delta \lambda_0}} \frac{\delta \lambda_0}{L_i}\bigg)
	\bigg(\prod_{i=1}^d \mu_i \bigg)
	\prod_{k=1}^{b-1}	
	\bigg(\prod_{i\in\mathcal{I}_k} \mu_i \bigg).
\end{equation}
We now claim that, at the maximum, the $(\mu_i)$ either take the value $1$ or $\lambda_0$. To prove this claim, assume that, at the maximum, the $\mu_i$ take a number $n$ of distinct values $1 \leq a_1 < \dots < a_n \leq \lambda_0$. Replacing $L_i$ by $\mu_{f(i)}^2$ in the above expression, it takes the form
$$
\lambda_0 (\lambda_0 \delta)^{\alpha_0} \prod_{i=1}^n a_i^{\alpha_i}, \qquad \mbox{where $\alpha_i \in \mathbb{Z}$}.
$$
If $\alpha_i >0$ and $a_i < \lambda_0$, then this expression will increase if the value of $a_i$ is increased until $a_{i+1}$ or $\lambda_0$; and similarly, if $\alpha_i <0$ and $a_i > 1$, it will decrease if the value of $a_i$ is decreased until $a_{i-1}$ or $1$. This contradicts the maximality of $(\mu_i)$ unless $a_i$ only takes the values $\lambda_0$ or $1$ for $\alpha_i \neq 0$. There remains the case where $\alpha_i =0$, but then $\mu_i$ can be assigned the value $\lambda_0$ or $1$ indifferently.

\medskip
\noindent {}{}{\underline{Step 3: Maximizing in $(\mathcal{I}_k)$ and $(\sigma_k)$.}} We showed that the maximum of $\mathcal{F}$ is less than the maximum of~\eqref{functiontomaximize}, under the constraint that $\mu_i^2 \leq \max(L_{\sigma_k(i)+1},L_{k+1})$ if $i \in \mathcal{I}_k$ (with the convention that $L_{k} = 0$ for ${k \geq d+2}$); and under the further constraint that $L_i$ can only take the values $\delta \lambda_0,1,\lambda_0$, and $\mu_i$ can only take the values $1,\lambda_0$.

There are now two cases to consider:
\begin{itemize}
\item If \(k\leq \min (b-1,d)\) and $L_{k+1} = \lambda_0^2$, then the optimal choice for $\mathcal{I}_k$ is $\{ 1,\dots,d \}$.
\item Otherwise,  $\mathcal{I}_k$ should have the same cardinal as the set of $L_i$, $i \geq 2$, equal to $\lambda_0^2$, and $\sigma_k +1$ should map $\mathcal{I}_k$ to this set.
\end{itemize}

\medskip
\noindent {}{}{\underline{Step 4: Conclusion.}} As a result of the previous reductions, we find 
 \begin{multline*}
 	\max_{ \substack{\lambda_0^2 \sim L_1\geq \dotsb\geq L_{\min (b,d+1)} \geq 0 \\ \lambda_0 \geq\mu_1\geq\dotsb\geq \mu_d \geq 0}}
 	\max_{\substack{ \I_k \subset \{1,\dotsc,d\} \\ \sigma_k: \mathcal{I}_k \hookrightarrow \{1,\dots, d \} }}
 	\mathcal{F}((L_i),(\mu_i),(\mathcal{I}_k),(\sigma_k)) \lesssim
\\
\max_{\substack{
		\lambda_0=L_1\geq \dotsb \geq L_{\min (b,d+1)}
		\\
		\mu_1\geq \dotsb \geq\mu_d
		\\
		\\L_i \in\{\delta\lambda_0,1,\lambda_0^2\}, \mu_i\in\{1,\lambda_0\}}}
\lambda_0
\bigg(
\prod_{\substack{i=1,\dotsc,\min (b,d+1) \\ L_i>\delta \lambda_0}} \frac{\delta \lambda_0}{L_i}\bigg)
\bigg(\prod_{i=1}^d \mu_i \bigg)
\prod_{k=1}^{b-1}	
\bigg({\prod_{\substack{
i\in\{1,\dotsc,d\}
\\
L_{i+1}=\lambda_0^2,\text{ or}
\\
(k\leq d \text{ and }L_{k+1}=\lambda_0^2)
}}} \mu_i \bigg).
\end{multline*}

{}{}{Notice that we are assuming again that the $(\mu_i)$ and $(L_i)$ are ordered; a moment of reflection shows that this is possible since the permutation $(\sigma_k)$ can be freely chosen.}
This expression is visibly nondecreasing in $(\mu_i)$, so we might as well take all $\mu_i$ to be $\lambda_0$. 

In order to evaluate the resulting expression, we need to know the number of $L_k$ equal to, respectively, $\delta \lambda_0,1,\lambda_0$; this is also the information needed to determine $\mathcal{I}_k$; therefore, we define the numbers
\begin{align*}
	b_0 &= \#\{1\leq j\leq \min (b,d+1) : L_j = \delta\lambda_0\},
	\\	b_1 &= \#\{1\leq j\leq \min (b,d+1) : L_j = 1\},
	\\	b_2 &= \#\{1\leq j\leq \min (b,d+1) : L_j = \lambda_0^2\},
\end{align*}
which are such that
$$
b_0 + b_1 + b_2 = \min (b,d+1).
$$
Letting \(\chi=1\) if \({1>\delta\lambda_0}\) and \(\chi=0\) otherwise, the maximization procedure shows that the function to be optimized is bounded by
\begin{align*}
& \max_{b_0+b_1+b_2=\min (b,d+1)} \lambda_0 (\delta \lambda_0)^{b_1\chi + b_2} \lambda_0^{-2 b_2} \lambda_0^d \lambda_0^{(b_2-1) d} \lambda_0^{(b - b_2)(b_2-1)} \\
& \qquad \qquad = \max_{b_1+b_2\leq\min (b,d+1)} {(\delta \lambda_0)^{b_1\chi}\delta^{b_2}}  \lambda_0^{-b_2^2 + (b+d) b_2 + 1-b}
\end{align*}
We notice first that $b_1$ can be taken {}{}{to be zero}. Second, there remains to dispose of the assumption that all $\mu_i$ are {}{}{non-zero}, which was made at the beginning of Subsection~\ref{cmwps}. By the comments just prior to the start of Subsection~\ref{cmwps}, this equivalent to reducing the dimension $d$. But some thought shows that the above expression is increasing with  $d$, so that allowing for smaller $d$ is harmless. Overall, the final bound we find is
$$
\max_{ \substack{\lambda_0^2 \sim L_1\geq \dotsb\geq L_{\min (b,d+1)} \geq 0 \\ \lambda_0 \geq\mu_1\geq\dotsb\geq \mu_d \geq 0}}
\max_{\substack{ \I_k \subset \{1,\dotsc,d\} \\ \sigma_k: \mathcal{I}_k \hookrightarrow \{1,\dots, d \} }}
\mathcal{F}((L_i),(\mu_i),(\mathcal{I}_k),(\sigma_k))\lesssim
 \max_{b_2 \leq \min (b,d+1)}  \delta^{ b_2} \lambda_0^{-b_2^2 + (b+d) b_2 + 1-b}.
$$

\subsection{Borel-Cantelli and the end of the argument}

The three previous subsections give the estimate
$$
\int_{[1,2]^d}
	\int_{\lambda_0/2}^{\lambda_0}
	\lVert P_{\lambda,\delta} \rVert_{L^1\to L^\infty}^b
	\, d \lambda
	\, d {\beta'} \lesssim 
	(\log\lambda_0)^{bd+d}\max_{b_2 \leq \min (b,d+1)}  \delta^{ b_2} \lambda_0^{-b_2^2 + (b+d) b_2  + 1-b},
$$ 
valid uniformly for any off-diagonal coefficients \(\)
Since $P_{\lambda,\delta}$ varies on a scale $\sim \delta$, this implies that, for $\lambda \in [\frac{\lambda_0}{2}, \lambda_0]$,
$$
\int_{[1,2]^d} \sup_{\frac{\lambda_0}{2} < \lambda < \lambda_0} \lVert P_{\lambda,\delta} \rVert_{L^1\to L^\infty}^b \, d \beta'
\lesssim
(\log\lambda_0)^{bd+d}
 \delta^{-1}  \max_{b_2 \leq \min (b,d+1)}  \delta^{ b_2} \lambda_0^{-b_2^2 + (b+d) b_2 + 1-b}.
$$
By  Borel-Cantelli as in Appendix~\ref{appendix}, this implies that, for any $\epsilon > 0$,
\begin{equation}\label{eq:last-max}
	\lVert P_{\lambda,\delta} \rVert_{L^1\to L^\infty} \lesssim_{\beta,\epsilon} \delta^{-1/b} \lambda^{\epsilon} \left[ \max_{b_2 \leq \min (b,d+1)}  \delta^{ b_2} \lambda^{-b_2^2 + (b+d) b_2 + 1-b} \right]^{1/b},
\end{equation}
both for fixed off-diagonal coefficients and almost all \(\beta'\in[1,2]^d\), and also for almost all matrices \((\beta_{ij})\) with \(\beta'\in[1,2]^d\) and small off-diagonal coefficients. This proves \eqref{eq:thm-sharp}.

We now observe that if \(b_2\) is strictly less than \(\min (b,d+1)\), then
\begin{align*}
\frac{ \delta^{ b_2+1} \lambda^{-(b_2+1)^2 + (b+d) (b_2+1) +1-b} }
{
 \delta^{ b_2} \lambda^{-b_2^2 + (b+d) b_2 + 1-b} 
}
& =
\delta
\lambda^{-2b_2-1+b+d}
\\
&\geq
\delta
\lambda^{-2\min (b-1,d)-1+b+d}.
\end{align*}
This will be \(\geq 1\) provided that \(\delta\geq\lambda^{\min (b-d-1,d+1-b)}\), and so for such \(\delta\) the maximum in \eqref{eq:last-max} is reached for $b_2 = \min (b,d+1)$. We will therefore impose the condition  \(\delta\geq\lambda^{\min (b-d-1,d+1-b)}\) for convenience rather than because we believe it to be optimal. This yields
$$
\lVert P_{\lambda,\delta} \rVert_{L^1\to L^\infty} \lesssim_{\epsilon,\beta'}
\begin{cases}
\delta^{1- \frac 1 b} \lambda^{d - 1 + {\frac{1}{b}} + \epsilon}
&(b\leq d, \delta\geq\lambda^{b-d-1}),
\\
{\delta^{\frac d b} \lambda^{ d -\frac{d}{b}+ \epsilon}}
&(b\geq d+1, \delta\geq\lambda^{d+1-b}).
\end{cases}
$$
We use the first of these alternatives when \(\delta\geq\lambda^{1-d} \), and otherwise we use the second. In particular by writing \(a= d+1-b\) if \(\delta\geq\lambda^{1-d}\) and \(a = b-d-1\) otherwise, we obtain \eqref{eq:thm-blunt}. \qed

\appendix

\section{Borel-Cantelli lemma}\label{appendix}

{}{}{
	At several points we use versions of the following argument to turn a bound on average over \((\beta_{ij})\) into a bound for almost all \(\beta\).
}

{}{}{
	Let \(\beta=(\beta_{ij})\) be as in Definition~\ref{defgeneric}. That is, either \(\beta_{ij}=\delta_{ij}\beta_i\) with \(\beta_i\in[1,2]\) for each \(i\); or $\beta_{ij} = \delta_{ij} + h_{ij}$ for each \(1\leq i,j \leq d\) and some $h_{ij}=h_{ji} \in [-\frac{1}{10d^2} , \frac{1}{10d^2} ]$.}
	
	{}{}{
	We consider a function \(\Phi( \beta, \vec{P} )\) that depends on \((\beta_{ij})\) and also a list of parameters \(P_1,\dotsc,P_k\), for some \(k\in\N\). In applications these parameters will be integer powers of 2, constrained by some {}{}{inequalities}.
	For example in the proof of Lemma~\ref{generic} we should
let \(\vec{P}\) take values in the set
	\[
	S=\{
	(N,T, Q_1,\dotsc,Q_d)\ \in (2^{\Z})^{d+2} : 
	2\leq Q_j \leq N, \tfrac{1}{4N}\leq T \leq N^\kappa
	\},
	\]
	and put
	\[
	\Phi( \beta, \vec{P} )=
	\int_{1}^{2} \dots \int_{1}^{2} \int_0^T \Lambda_{Q_1}(t) \Lambda_{Q_2}(\beta_2 t) \dots \Lambda_{Q_k}(\beta_k t)\,dt \,d\beta_2 \dots d\beta_k.
	\]
	\begin{lem}\label{borelcantelli}
	Let \(\Phi\) be as above. 
	Assume that \(\vec{P}\) takes values in a subset \(S\subset(2^{\Z})^{k}\) such that for any \(\epsilon>0\) we have
	\[
	\sum_{\vec{P}\in S}
	P_1^{-\epsilon}
	<\infty.
	\]
	If \(\Psi\) is some real function of \(\vec{P}\) with
	\[
	\int 
	|\Phi( \beta, \vec{P} )|
	\, d\beta
	\leq \Psi(\vec{P}),
	\]
	then for every \(\vec{P}\in S\), then for generic \(\beta\) we have
	\[
	\Phi( \beta, \vec{P} )
	\lesssim_{\beta,\epsilon}
	P_1^{\epsilon}
	 \Psi(\vec{P}).
	\]
\end{lem}
}{}{}{
The lemma follows at once from the standard statement of the Borel-Cantelli lemma.}


\bibliographystyle{plain}
\bibliography{references}
\end{document}